\pdfoutput=1
\RequirePackage{ifpdf}
\ifpdf 
\documentclass[pdftex]{sigma}
\else
\documentclass{sigma}
\fi

\numberwithin{equation}{section}

\newtheorem{Theorem}{Theorem}[section]
\newtheorem{Lemma}[Theorem]{Lemma}
\newtheorem{Proposition}[Theorem]{Proposition}
\newtheorem{Conjecture}[Theorem]{Conjecture}
 { \theoremstyle{definition}
\newtheorem{Definition}[Theorem]{Definition}
\newtheorem{Example}[Theorem]{Example}
\newtheorem{Remark}[Theorem]{Remark}
\newtheorem{Question}[Theorem]{Question}
\newtheorem{Exercise}[Theorem]{Exercise}
\newtheorem{Claim}[Theorem]{Claim}
\newtheorem*{Note}{Note}
}

\begin{document}

\allowdisplaybreaks

\newcommand{\arXivNumber}{1809.05747}

\renewcommand{\thefootnote}{}

\renewcommand{\PaperNumber}{035}

\FirstPageHeading

\ShortArticleName{An Introduction to Higgs Bundles via Harmonic Maps}

\ArticleName{An Introduction to Higgs Bundles via Harmonic Maps\footnote{This paper is a~contribution to the Special Issue on Geometry and Physics of Hitchin Systems. The full collection is available at \href{https://www.emis.de/journals/SIGMA/hitchin-systems.html}{https://www.emis.de/journals/SIGMA/hitchin-systems.html}}}

\Author{Qiongling LI}

\AuthorNameForHeading{Q.~Li}

\Address{Chern Institute of Mathematics and LPMC, Nankai University, Tianjin 300071, China}
\Email{\href{mailto:qiongling.li@gmail.com}{qiongling.li@gmail.com}}
\URLaddress{\url{https://sites.google.com/site/qionglingli/}}

\ArticleDates{Received October 16, 2018, in final form April 26, 2019; Published online May 04, 2019}

\Abstract{This survey studies equivariant harmonic maps arising from Higgs bundles. We explain the non-abelian Hodge correspondence and focus on the role of equivariant harmonic maps in the correspondence. With the preparation, we review current progress towards some open problems in the study of equivariant harmonic maps.}

\Keywords{Higgs bundles; harmonic maps; non-abelian Hodge correspondence}

\Classification{53C43; 53C07; 53C21}

\renewcommand{\thefootnote}{\arabic{footnote}}
\setcounter{footnote}{0}

\section{Introduction}
In this survey, we study equivariant harmonic maps from the Riemannian universal cover of a~surface $S$ into the noncompact symmetric space for semisimple representations from the fundamental group of $S$ into a semisimple Lie group.

The celebrated non-abelian Hodge correspondence, developed mainly by Corlette \cite{Corlette}, Do\-naldson~\cite{Donaldson}, Hitchin \cite{Hitchin92} and Simpson \cite{Simpson88}, is a homeomorphism between the moduli space of Higgs bundles and the representation variety. Equivariant harmonic maps play an important role in this correspondence. Let's elaborate this correspondence in more detail. Following the work of Donaldson \cite{Donaldson} and Corlette~\cite{Corlette}, for any irreducible representation $\rho$ of the fundamental group of~$S$ into a semisimple Lie group~$G$, there exists a unique $\rho$-equivariant harmonic map~$f$ from $\widetilde\Sigma$ to the corresponding symmetric space of~$G$. The equivariant harmonic map further gives rise to a Higgs bundle, a pair $(E,\phi)$ consisting of a holomorphic vector bundle $E$ over a~Riemann surface structure $\Sigma$ on $S$ and a holomorphic section of $\operatorname{End}(E)\otimes K$, the Higgs field, and $K$ is the holomorphic cotangent line bundle over $\Sigma$. Conversely, by the work of Hitchin~\cite{Hitchin92} and Simpson~\cite{Simpson88}, a stable Higgs bundle admits a unique harmonic metric on the bundle solving the Hitchin equation. The harmonic metric further gives rise to an irreducible representation $\rho$ into $G$ and a $\rho$-equivariant harmonic map into the corresponding symmetric space. These two directions together give the celebrated non-abelian Hodge correspondence.

The relation between harmonic maps with Higgs bundles is transcendental since it involves solving a highly nontrivial second-order elliptic system, the Hitchin equation. Our goal is to make use of Higgs bundles and the solution to the Hitchin equation to investigate the properties of corresponding harmonic maps. For instance, we ask how the energy density of harmonic maps changes along the $\mathbb C^*$-flow on the moduli space of Higgs bundles.

The paper is organized as follows. In Part~\ref{Part1}, we recall some preliminaries on Higgs bundles and the non-abelian Hodge correspondence. Since we focus on the harmonic map point of view, our main goal is to explain the role of harmonic maps (or harmonic metrics) in the correspondence and its explicit relationship with the data of Higgs bundles. Then we introduce several important concepts for the moduli space of Higgs bundles including the Hitchin fibration, the Hitchin section and the $\mathbb C^*$-action, the notion of cyclic Higgs bundles and the maximal representations. In Part~\ref{Part2}, we write the Hitchin equation explicitly and express the energy density, the pullback metric and the sectional curvature of the tangent plane about harmonic maps. Then we use explicit examples to do calculations on the Hitchin equation and relate the estimates with the associated geometry. In the last Part \ref{Part3}, we discuss selected topics on equivariant harmonic maps in terms of different types of information of Higgs bundles. We collect several open and interesting questions here and explain to the reader the current progress towards such questions.

This survey is based on lecture notes prepared for the 3-hour mini-course ``An introduction to cyclic Higgs bundles and complex variation of Hodge structures" that the author gave at the University of Illinois at Chicago. The paper particularly deals with the analytic aspect related to the non-abelian Hodge correspondence. The readers might find that it does not mention at all the proofs of big theorems, such as the theorem of Hitchin and Simpson or the theorem of Corlette and Donaldson. Instead, the author puts more time on introducing basic notations in differential geometry, deducing expressions of associated geometric objects, and also doing detailed calculations on the Hitchin equation. All of these efforts are aimed at helping the readers to make use of the tool of Higgs bundles by understanding the solutions to the Hitchin equation and possibly the related geometry. This survey is targeted at graduate students and junior postdocs.

\part{Set-up}\label{Part1}
\section{Preliminaries and non-abelian Hodge correspondence}\label{One}
Notations:
\begin{alignat*}{3}
& S \quad && \text{-- a smooth surface};& \\
& \Sigma \quad && \text{-- a Riemann surface structure over $S$};& \\
& E \quad && \text{-- a complex vector bundle on $\Sigma$};& \\
& D \quad && \text{-- a connection on $E$};& \\
& H \quad && \text{-- a Hermitian metric on $E$};& \\
& \bar\partial_E \quad && \text{-- a holomorphic structure on $E$};& \\
& \nabla_{\bar\partial_E,H} \quad && \text{-- the Chern connection determined by $\bar\partial_E$ and $H$};& \\
& K \quad && \text{-- the holomorphic cotangent bundle of $\Sigma$};& \\
& g_0 \quad && \text{-- a conformal metric on $\Sigma$};& \\
& \omega \quad && \text{-- the K\"ahler form on $(\Sigma, g_0)$};& \\
& O \quad && \text{-- the trivial line bundle $S\times \mathbb C$ on $S$};& \\
& \mathcal{O} \quad && \text{-- the trivial holomorphic line bundle $\Sigma\times \mathbb C$ on $\Sigma$};& \\
& \Gamma(S, E) \quad && \text{-- the space of smooth sections of $E$};& \\
& \Omega^k(S, E) \quad && \text{-- the space of smooth $k$-forms valued in $E$}. &
\end{alignat*}

\subsection{Basic notions}
Throughout the paper, let $S$ be a closed orientable surface of genus $g$ at least $2$ and fix a~Riemann surface structure $\Sigma$ on $S$. We begin with a rapid introduction to differential geometry of complex vector bundles. One may refer the materials to Kobayashi's book~\cite{Kobayashi}. Note that the following notions can be defined for higher-dimensional manifolds. But for simplicity, we only deal with surfaces and one may notice our definitions might be simpler than the ones for higher-dimensional manifolds in some cases.
\begin{Definition}A Hermitian metric $H$ on a complex vector bundle $E$ over $S$ is a $C^{\infty}$ family of Hermitian inner products on $E$, which is $\mathbb C$-linear in the second variable and conjugate-linear in the first variable.
\end{Definition}

One can extend the definition of the Hermitian pairing between $\Omega^k(\Sigma, E)$ and $\Omega^l(\Sigma, E)$ by pairing sections in $E$ and wedging the forms as usual: for $\alpha_1\in \Omega^k(\Sigma, \mathbb C)$, $\alpha_2\in \Omega^l(\Sigma, \mathbb C)$, and $s_1, s_2\in \Gamma(\Sigma, E)$,
\begin{gather*}
H(\alpha_1\otimes s_1, \alpha_2\otimes s_2)=(\alpha_1\wedge\alpha_2)\cdot H(s_1,s_2)\in \Omega^{k+l}(\Sigma, \mathbb C).
\end{gather*}

Suppose we are given two Hermitian bundles $(E_1, H_1)$ and $(E_2, H_2)$, consider a section $\psi\in\Omega^k(S, \operatorname{Hom}(E_1,E_2))$, then we can define its adjoint $\psi^*\in \Omega^k(S, \operatorname{Hom}(E_2, E_1))$ by the property: for $s_1\in \Gamma(S, E_1), s_2\in \Gamma(S, E_2)$,
\begin{gather}\label{dual}
H_2(\psi s_1, s_2)=H_1(s_1, \psi^* s_2).\end{gather}

Fix any background K\"ahler metric $g_0=g_0(z)({\rm d}z\otimes {\rm d}\bar z+{\rm d}\bar z\otimes {\rm d}z)$ on $\Sigma$ where $\omega={\rm i} g_0(z){\rm d}z\wedge {\rm d}\bar z$ is the K\"ahler form which is also the volume form. We renormalize the metric $g_0$ such that $\int_\Sigma \omega=2\pi$. The metric $g_0$ induces a natural pairing $\langle\,,\,\rangle$ on $\Omega^k(\Sigma,\mathbb C)$. The Hodge star $\star$ is a~conjugate-linear map from $\Omega^k(\Sigma, \mathbb C)$ to $\Omega^{2-k}(\Sigma, \mathbb C)$ such that
\begin{gather}\label{star}
\alpha_1\wedge\star \alpha_2=\langle \alpha_1,\alpha_2\rangle \omega.
\end{gather} In particular, $\star {\rm d}z={\rm i}\,{\rm d}\bar z$ and $\star {\rm d}\bar z=-{\rm i}\,{\rm d}z$.

Combining the formula in equation (\ref{dual}) and (\ref{star}), one can extend the definition of the Hodge star operator to $\Omega^k(\Sigma, \operatorname{Hom}(E_1, E_2))$ as follows: for $\Psi=\sum\limits_{i=1}^k\alpha_i\otimes\psi_i\in \Omega^k(\Sigma, \operatorname{Hom}(E_1, E_2))$ where $\alpha_i\in \Omega^k(\Sigma, \mathbb C)$, $\psi_i\in \Gamma(\Sigma, \operatorname{Hom}(E_1,E_2))$, we define
\begin{gather}\label{dualstar}
\star\Psi=\star\left(\sum\limits_{i=1}^k\alpha_i\otimes\psi_i\right):=\sum\limits_{i=1}^k(\star\alpha_i)\otimes \psi_i^*\in \Omega^{2-k}(\Sigma, \operatorname{Hom}(E_2, E_1)).
\end{gather}

\begin{Note}
For $\phi\in\Omega^{1,0}(\Sigma, \operatorname{Hom}(E_1, E_2))$, $\star\phi={\rm i}\phi^*$. For $\phi\in\Omega^{0,1}(\Sigma, \operatorname{Hom}(E_1, E_2))$, $\star\phi=-{\rm i}\phi^*$. For $\Psi\in \Omega^1(\Sigma, \operatorname{Hom}(E_1,E_2))$, by decomposing into $(1,0)$- and $(0,1)$-forms, $\Psi=\Psi^{1,0}+\Psi^{0,1}$. Then $\star\Psi={\rm i}(\Psi^{1,0})^*-{\rm i}(\Psi^{0,1})^*$.
\end{Note}

The induced pairing on $\Omega^k(\Sigma, \operatorname{Hom}(E_1,E_2))$ is linearly extending the pairing
\begin{gather*}
\langle \alpha_1\otimes \psi_1, \alpha_2\otimes\psi_2\rangle :=\langle \alpha_1,\alpha_2\rangle \cdot \langle \psi_1,\psi_2\rangle =\frac{\alpha_1\wedge\star\alpha_2}{\omega}\cdot \operatorname{tr}(\psi_1\psi_2^*)
\end{gather*}
for $\alpha_1\otimes\psi_1, \alpha_2\otimes\psi_2\in \Omega^1(\operatorname{Hom}(E_1, E_2))$.

Equivalently, we have the pairing for $\Psi_1, \Psi_2\in \Omega^1(\Sigma, \operatorname{Hom}(E_1, E_2))$,
\begin{gather}\label{GeneralizedPairing}
\langle \Psi_1, \Psi_2\rangle =\operatorname{tr}(\Psi_1\wedge\star\Psi_2)/\omega,
\end{gather}
where $tr$ denotes the trace for an endomorphism, or likewise
\begin{gather*} ||\Psi||^2=\big(||\Psi(\partial_x)||^2+||\Psi(\partial_y)||^2\big)/g_0(z).\end{gather*}

\begin{Definition}A connection on a complex vector bundle $E$ over $S$ is a differential operator $D\colon \Omega^k(S,E)\rightarrow \Omega^{k+1}(S,E)$ satisfying the Leibniz rule: if $\alpha\in \Omega^p(S, \mathbb C)$, $\sigma\in \Omega^{k}(S,E)$,
\begin{gather*} D(\alpha\wedge\sigma)={\rm d}\alpha\wedge\sigma+(-1)^p\alpha\wedge D\sigma.\end{gather*}
\end{Definition}
For example, on a trivial vector bundle $S\times \mathbb C^n$, the usual differential operator ${\rm d}$ is a connection. Therefore, for a~rank~$n$ complex vector bundle~$E$ with trivial determinant, we can define a~${\rm SL}(n,\mathbb C)$-connection as follows.
\begin{Definition}On a complex vector bundle $E$ satisfying $\det E\cong O$, a ${\rm SL}(n,\mathbb C)$-connection on~$E$ is a connection such that its induced connection on the trivial line bundle $\det E$ is~${\rm d}$.
\end{Definition}

\begin{Definition}The curvature of a connection $D$ on $E$ is the operator
\begin{gather*}F_D=D\circ D\colon \ \Omega^k(S,E)\rightarrow \Omega^{k+2}(S,E).\end{gather*}
\end{Definition}
Fact: $F_D$ turns out to be $C^{\infty}$-linear, i.e., $F_D\in \Omega^2(S,\operatorname{End}(E))$. The Chern--Weil theory tells us that the first Chern class of $E$ is $c_1(E)=\big[\frac{\rm i}{2\pi}\cdot \operatorname{tr}(F_D)\big]\in H_{\rm dR}^2(S, \mathbb C)$, which does not depend on the choices of the connection~$D$. The degree of~$E$ is the integral of any representative in~$c_1(E)$, that is,
\begin{gather*}
\deg E=\int_S c_1(E)=\int_S \frac{{\rm i}}{2\pi}\cdot \operatorname{tr}(F_D)=\int_S \frac{{\rm i}}{2\pi}\cdot \Lambda \operatorname{tr}(F_D)\cdot \omega.
\end{gather*}

The contraction operator $\Lambda\colon \Omega^2(\Sigma, \mathbb C)\rightarrow \Omega^0(\Sigma, \mathbb C)$ is defined by $\Lambda(f\omega)=f$, for any smooth function $f$ on $\Sigma$. We extend the contraction operator to $\Omega^2(\Sigma, \operatorname{End}(E))$. Equivalently, we can write it as $\Lambda (F_D)=F_D/\omega$.

A connection $D$ is said to be \textit{flat} if $F_D=0$. Fix a basepoint $p\in S$ and a frame $e$ of $E_p$. Denote $\pi_1(S)=\pi_1(S, p)$. A flat connection $D$ on $E$ gives rise to a representation of $\pi_1(S)$ as follows. For each loop based at $p$, the parallel transport of the frame $e$ defines an element of ${\rm GL}(n, \mathbb C)$. In particular, for a flat connection, the element only depends on the homotopy class of the loop. Therefore, we obtain an element $\rho=\operatorname{hol}(D)\in \operatorname{Hom}(\pi_1, {\rm GL}(n, \mathbb C))$. If $D$ is a flat ${\rm SL}(n,\mathbb C)$-connection, the holonomy lies in ${\rm SL}(n,\mathbb C)$ correspondingly. Conversely, given a~representation $\rho\colon \pi_1(S)\rightarrow {\rm SL}(n, \mathbb C)$, we can construct a flat vector bundle~$(E, D)$ as follows,
\begin{gather*}(E,D):= \big(\widetilde{S}\times_{\rho}\mathbb C^n, \text{the natural connection descends from ${\rm d}$ on $\widetilde{S}\times \mathbb C^n$}\big),\end{gather*} where $\widetilde{S}$ is the universal cover of $S$.
\begin{Definition}\quad
\begin{enumerate}\itemsep=0pt
\item[(1)] A connection $D$ on $E$ is called irreducible if there exists no proper $D$-invariant subbundle.
\item[(2)] A connection $D$ is called reductive if $(E, D)=\bigoplus\limits_{i=1}^k(E_i, D_i)$ where each $D_i$ is an irreducible connection on $E_i$.
\end{enumerate}
\end{Definition}
Correspondingly, we have the following definitions.
\begin{Definition}\quad
\begin{enumerate}\itemsep=0pt
\item[(1)] A representation $\rho\colon \pi_1(S)\rightarrow {\rm SL}(n,\mathbb C)$ is called irreducible if the induced representation on~$\mathbb C^n$ is irreducible.
\item[(2)] A representation $\rho\colon \pi_1(S)\rightarrow {\rm SL}(n,\mathbb C)$ is called reductive if the induced representation on~$\mathbb C^n$ is completely reducible.
\end{enumerate}
\end{Definition}
\begin{Definition}
A connection $D$ on $E$ is called unitary if for any two sections $s, t\in \Gamma(S,E)$,
\begin{gather*}{\rm d}(H(s,t))=H(Ds,t)+H(s,Dt).\end{gather*}
\end{Definition}

\begin{Definition}A holomorphic structure on a complex vector bundle $E$ over $\Sigma$ is a differential operator
$ \bar\partial_E\colon \Omega^{p,q}(\Sigma,E)\rightarrow \Omega^{p,q+1}(\Sigma,E)$ satisfying the Leibniz rule: if $\alpha\in \Omega^{p,q}(\Sigma, \mathbb C)$, $\sigma\in \Omega^{k,l}(\Sigma,E)$,
\begin{gather*}\bar\partial_E(\alpha\wedge \sigma)=(\bar\partial\alpha)\wedge\sigma+(-1)^{p+q}\alpha\wedge \bar\partial_E\sigma.\end{gather*}

We call a section $\sigma$ of $E$ holomorphic if $\bar\partial_E\sigma=0$.
\end{Definition}
For example, on a trivial vector bundle $\Sigma\times \mathbb C^n$, the usual differential operator $\bar\partial$ is a holomorphic structure.

Given any connection $D$ on $E$, by decomposing into $(1,0)$- and $(0,1)$-forms, we have $D=D^{1,0}+D^{0,1}$. Then $D^{0,1}$ gives a holomorphic structure on $E$. But given a holomorphic structure on $E$, there are many connections $D$ such that $D^{0,1}=\bar\partial_E$.

\begin{Theorem}For a holomorphic vector bundle $E$ with a Hermitian metric $H$, there exists a~unique connection $\nabla_{\bar\partial_E, H}$, called the Chern connection, such that
\begin{enumerate}\itemsep=0pt
\item[$(i)$] $\nabla_{\bar\partial_E, H}^{0,1}=\bar\partial_E$,
\item[$(ii)$] $\nabla_{\bar\partial_E, H}$ is unitary.
\end{enumerate}
\end{Theorem}

The above conditions (i) and (ii) for the Chern connection $\nabla_{\bar\partial_E, H}$ imply that the following holds
\begin{gather}\label{HolomorphicChern}
\partial(H(s,t))=H(\bar\partial_E s,t)+H\big(s,\nabla_{\bar\partial_E, H}^{1,0}t\big).
\end{gather}

\subsubsection{The Riemannian geometry of symmetric space}
Denote $G={\rm SL}(n,\mathbb C)$, $K={\rm SU}(n)$ and by $\mathfrak g$, $\mathfrak k$ the corresponding Lie algebras $\mathfrak{sl}(n,\mathbb C)$, $\mathfrak{su}(n)$ respectively. With respect to the Killing form on $\mathfrak g$, we have an orthogonal decomposition $\mathfrak g=\mathfrak k\oplus \mathfrak p$, where $\mathfrak p={\rm i}\cdot \mathfrak{su} (n)$. The tangent space at $T_{eK}G/K$ is isomorphic to $\mathfrak p$. The Killing form $B$ on $\mathfrak g$ is \begin{gather*}B(Y_1,Y_2)=2n\cdot \operatorname{tr}(Y_1Y_2).\end{gather*} The restriction of $B$ on $T_{eK}G/K\cong\mathfrak p$. Denote by $L_g$ the left action by~$g$ on $G/K$. Pulling back the inner product on $T_{eK}G/K$ using $L_{g^{-1}}$, we can define a metric on $T_{gK}G/K$ . This is the unique $G$-invariant metric on $G/K$ up to a scalar multiple.

The sectional curvature of the tangent plane spanned by two tangent vectors $Y_1, Y_2\in \mathfrak p$ is given by (see Jost's book \cite{Jost} for reference)
\begin{gather}\label{CurvatureOnp}
K(Y_1,Y_2)=\frac{B([Y_1,Y_2], [Y_1, Y_2])}{B(Y_1,Y_1)B(Y_2,Y_2)-B(Y_1,Y_2)^2}\leq 0.
\end{gather}
And the sectional curvature for a plane inside $T_{gK}G/K$ is just the sectional curvature of the pullback tangent plane inside $T_{eK}G/K\cong \mathfrak p$ by $L_{g^{-1}}$. Note that in the case $n=2$, ${\rm SL}(2,\mathbb C)/{\rm SU}(2)$ is of constant sectional curvature $-\frac{1}{2}$.

We consider a model for the space $G/K$, the space of positive definite Hermitian matrices of unit determinant \begin{gather*}N=\big\{ A\in M_n(\mathbb C)\,|\,\bar{A}^t=A,\, \det A=1,\, A>0\big\}.\end{gather*} The space $N$ may also be interpreted as the space of Hermitian metrics on $\mathbb C^n$ inducing the metric $1$ on $\det \mathbb C^n$. The group ${\rm SL}(n, \mathbb C)$ acts transitively on $N$ on the left, $g\cdot A=\big(g^{-1}\big)^*Ag^{-1}$ for $A\in N$ and $g\in {\rm SL}(n,\mathbb C)$. The map $\Psi\colon G/K\ni gK\rightarrow g\cdot {\rm Id}=\big(g^{-1}\big)^*g^{-1}\in N$ defines a~diffeomorphism which is equivariant for the left action of~$G$, that is, $L_g\circ \Psi=\Psi\circ L_g$ where~$L_g$ denotes the left action by $g$ on both~$G/K$ and~$N$.

We equip $N$ with the Riemannian metric from the one on $G/K$ using the diffeomorphism $\Psi$. Precisely, the pairing at $T_AN$ is by pulling back the pairing on $\mathfrak p$ using the map $L_{g^{-1}}\circ \Psi^{-1}$ and so is the sectional curvature of plane. At the point $A=g\cdot {\rm Id}=\big(g^{-1}\big)^*g^{-1}\in N$ for some $g\in G$, the differential map ${\rm d}\big(L_{g^{-1}}\circ\Psi^{-1}\big)|_{A}\colon T_AN\rightarrow \mathfrak p$ is given by: for every tangent vector $M\in T_AN$,
\begin{align}
{\rm d}\big(L_{g^{-1}}\circ\Psi^{-1}\big)\big|_{A}(M)& ={\rm d}\big(\Psi^{-1}\circ L_{g^{-1}}\big)\big|_{A}(M)={\rm d}\big(\Psi^{-1}\big)\big|_{{\rm Id}}\big(g^*Mg\big)\nonumber\\
&=-\tfrac{1}{2}g^*Mg = -\tfrac{1}{2}A{\rm d}\big(g^{-1}\big)\big(A^{-1}M\big).\label{PullbackVector}
\end{align}

Denote the metric on $N$ by $g_N$, then at $A=g\cdot {\rm Id}=\big(g^{-1}\big)^*g^{-1}\in N$, the metric $g_N$ is given by, for $M_1, M_2\in T_AN$,
\begin{align}
g_N(M_1,M_2) & = B\big({\rm d}\big(L_{g^{-1}}\circ\Psi^{-1}\big)\big|_{A}(M_1), {\rm d}\big(L_{g^{-1}}\circ\Psi^{-1}\big)\big|_{A}(M_2)\big)\nonumber\\
&= B\big(A{\rm d}\big(g^{-1}\big)\big({-}\tfrac{1}{2}A^{-1}M_1\big), A{\rm d}\big(g^{-1}\big)\big({-}\tfrac{1}{2}A^{-1}M_2\big)\big)=\tfrac{1}{4}B\big(A^{-1}M_1, A^{-1}M_2\big)\nonumber\\
&= \tfrac{n}{2}\operatorname{tr}\big(A^{-1}M_1A^{-1}M_2\big).\label{Metric}
\end{align}

\subsection{The non-abelian Hodge correspondence}
Suppose we are given a representation $\rho\colon \pi_1\rightarrow {\rm SL}(n,\mathbb C)$ and hence a flat ${\rm SL}(n,\mathbb C)$-bundle $(E,D)$ over $S$, and a Riemann surface structure~$\Sigma$, we aim to obtain the following holomorphic object, the Higgs bundle. This is one direction of the non-abelian Hodge correspondence.

\begin{Definition}A rank $n$ Higgs bundle over $\Sigma$ is a pair $(E, \phi)$ where $E$ is a holomorphic vector bundle of rank~$n$, and $\phi\in H^0(\Sigma,\operatorname{End}(E)\otimes K)$, called the \textit{Higgs field}. A~${\rm SL}(n,\mathbb C)$-Higgs bundle is a Higgs bundle $(E,\phi)$ satisfying $\det E=\mathcal{O}$ and $\operatorname{tr} \phi=0$.
\end{Definition}

\subsubsection{Harmonic metric}
We will make use of the following fact:

A connection $D$ on a Hermitian bundle $(E,H)$ decomposes uniquely as \begin{gather*}D=D_H+\Psi_H,\end{gather*} where
(1) $D_H$ is a unitary connection; and (2) $\Psi_H\in \Omega^1(\Sigma, \operatorname{End}(E))$ is self-adjoint. This decomposition is achieved by choosing $\Psi_H\in \Omega^1(\Sigma, \operatorname{End}(E))$ such that \begin{gather*}H(\Psi_Hs,t)=\tfrac{1}{2}\{H(D s,t)+H(s,D t)-{\rm d}(H(s,t))\}.\end{gather*}

We aim to choose the ``best" $H$. For a fixed flat ${\rm SL}(n,\mathbb C)$-vector bundle $(E, D)$ and a conformal Riemannian metric $g_0$ on~$\Sigma$, we define a functional on the space of Hermitian metrics on~$E$: \begin{gather}\label{E(H)}E(H)=\int_\Sigma \langle \Psi_H, \Psi_H\rangle \omega,\end{gather} where the pairing is defined in equation~(\ref{GeneralizedPairing}).
\begin{Definition}
A Hermitian metric $H$ on $(E,D)$ is called \textit{harmonic} if it is a critical point of~$E(H)$. Equivalently, $D_H(\star\Psi_H)=0$, where $\star$ is the Hodge star operator defined in equation~(\ref{dualstar}).
\end{Definition}
\begin{Remark}The functional $E(H)$ is invariant under a conformal change of the metric on the surface $\Sigma$. Therefore a harmonic metric on $(E, D)$ is well-defined on a Riemann surface.
\end{Remark}

\begin{Theorem}[Corlette~\cite{Corlette}, Donaldson \cite{Donaldson}] If $D$ is a reductive flat ${\rm SL}(n, \mathbb C)$-connection on~$E$ over~$\Sigma$, then there exists a harmonic metric~$H$ on~$E$ whose induced metric $\det H$ on $\det E\cong O$ is~$1$.

If $D$ is irreducible, then the harmonic metric is unique.
\end{Theorem}

We will first recall the definition of equivariant harmonic maps and then explain how a~harmonic metric on a flat vector bundle $(E, D)$ gives rise to an equivariant harmonic map.
\subsubsection{Harmonic map}
Fix $M$ a Riemannian manifold and $g_0$ a Riemannian metric on $S$. We consider a representation $\rho\colon \pi_1(S)\rightarrow \operatorname{Isom}(M)$, the isometry group of $M$. A map $f\colon \widetilde S\rightarrow M$ is called \textit{$\rho$-equivariant} if $f(\gamma\cdot x)=\rho(\gamma)\cdot f(x)$ for all $\gamma\in \pi_1(S)$ and $x\in \widetilde S$. Given a $\rho$-equivariant map $f\colon \big(\widetilde{S}, \widetilde g_0\big)\rightarrow M$ between two Riemannian manifolds, then ${\rm d}f\in \Gamma\big(\widetilde{S},T^*\widetilde{S}\otimes f^*TM\big)$ is also $\pi_1(S)$-equivariant. This implies that the function\begin{gather}\label{EnergyDensity}e(f)=\tfrac{1}{2}\langle {\rm d}f,{\rm d}f\rangle \colon \ \widetilde{S}\rightarrow \mathbb R\end{gather} is $\pi_1(S)$-invariant and hence descends to~$S$. We call $e(f)$ the \textit{energy density} on $\widetilde{S}$ and also on~$S$. The \textit{energy} $E(f)$ is the integral of $e(f)$ with respect to the volume form of $\operatorname{dvol}_{g_0}$, that is,
\begin{gather}\label{Energy}
E(f)=\int_{S} e(f) \operatorname{dvol}_{g_0}.
\end{gather}
Note that $E(f)$ is finite since $S$ is compact.
\begin{Definition}\label{Harmonic}The map $f$ is \textit{harmonic} if it is a critical point of the energy functional $E(f)$.

Equivalently, $f$ is harmonic if $\operatorname{tr}_{g_0}\nabla {\rm d}f=0$ where~$\nabla$ is the natural connection on $T^*S\otimes f^*TM$ induced by the Levi-Civita connections on $(S, g_0)$ and~$M$.
\end{Definition}

\begin{Remark}As in the case of $E(H)$, the energy $E(f)$ is invariant under the conformal change of the metric on~$S$. Hence we can talk about equivariant harmonic maps from the universal cover of a Riemann surface~$\Sigma$. However, the energy density still varies under the conformal change of the metric on $S$ and when we talk about the energy density, we will usually choose~$g_0$ as the conformal metric on~$\Sigma$.
\end{Remark}

\subsubsection{Equivalence between harmonic metric and harmonic map}
 A Hermitian metric on $E=\widetilde{S}\times_{\rho}\mathbb C^n$ inducing the metric $1$ on $\det E$ is a $\rho$-equivariant metric $H$ on $\widetilde{S}\times \mathbb C^n$ of unit determinant. Equivalently, it is a map $f\colon \widetilde{S}\rightarrow N\subset M_n(\mathbb C)$ satisfying
\begin{gather*}f(m)=\overline{\rho(\gamma)}^tf(\gamma\cdot m)\rho(\gamma), \qquad \forall\, \gamma\in \pi_1S, \quad m\in \widetilde{\Sigma}\end{gather*} by $H_m(s,t)=\bar s^tf(m)t$, for any two sections $s$, $t$ of $\widetilde{\Sigma}\times \mathbb C^n$. Using the statement $\Psi_H=-\frac{1}{2}f^{-1}{\rm d}f$ in Lemma~\ref{TangentHiggs} and the formula~(\ref{Metric}) of the metric $g_N$, we obtain
\begin{gather}\label{EnergyEqual}
e(f)=\tfrac{1}{2}||{\rm d}f||^2=\tfrac{1}{2}\big(||f_x||^2+||f_y||^2\big)/g_0=n\cdot \langle \Psi_H,\Psi_H\rangle ,
\end{gather}
since $||f_x||^2=\frac{n}{2}\operatorname{tr}\big(f^{-1}f_xf^{-1}f_x\big)=2n\cdot \operatorname{tr}(\Psi_H(\partial_x)\Psi_H(\partial_x))=2n\cdot ||\Psi_H(\partial_x)||^2$, and similarly $||f_y||^2=2n\cdot ||\Psi_H(\partial_y)||^2$.

Comparing equation (\ref{E(H)}) with~(\ref{Energy}), we can see that $E(f)=n\cdot E(H)$ following from equation~(\ref{EnergyEqual}). Finally, we see that the Hermitian metric~$H$ being harmonic (minimizing the functional $E(H)$) is equivalent to $f\colon \big(\widetilde{S},\widetilde{g_0}\big)\rightarrow N$ being harmonic (minimizing the energy of~$f$).

 \begin{Lemma}\label{TangentHiggs}
\begin{gather*}\Psi_H=-\tfrac{1}{2}f^{-1}{\rm d}f.\end{gather*}
\end{Lemma}
The following proof is taken from the lecture notes of O.~Guichard \cite{Guichard}.

\begin{proof}Firstly, $f^{-1}{\rm d}f\in \Omega^1\big(\widetilde{S},\operatorname{End}(\mathbb C^n)\big)$ is equivariant under the $\pi_1(S)$-action. Hence it descends to $S$ and is an element of $\Omega^1(S, \operatorname{End}(E))$.

For any two sections $s$, $t$ of $\widetilde{S}\times \mathbb C^n$, we have by definition, $H_m(s,t)=\bar s^tf(m)t$. We obtain
\begin{enumerate}\itemsep=0pt
\item[(i)] ${\rm d}(H(s,t))=H(D_H s,t)+H(s,D_Ht)$,
\item[(ii)] ${\rm d}(H(s,t))={\rm d}\big(\bar s^tft\big)={\rm d}\bar s^t\cdot f\cdot t+\bar s^t\cdot {\rm d}f\cdot t+\bar s^t\cdot f\cdot {\rm d}t=H({\rm d}s,t)+\bar s^t\cdot {\rm d}f\cdot t+H(s,{\rm d}t)$,
\item[(iii)] ${\rm d}=D_H+\Psi_H$ (lifted version).
\end{enumerate}
Combining (i), (ii) and (iii), we get
\begin{gather*} H(\Psi_Hs,t)+\bar s^t{\rm d}f t+H(s,\Psi_H t)=0\end{gather*} and hence $\Psi_H=-\frac{1}{2}f^{-1}{\rm d}f$.
\end{proof}

\subsubsection{From flat bundles to Higgs bundles}
Given a Hermitian metric $H$ on a flat bundle $(E,D)$, we have
\begin{align}
D&= D_H+\Psi_H \qquad \text{unitary + Herm}\nonumber\\
&= D_H^{1,0}+ D_H^{0,1}+\Psi_H^{1,0}+\Psi_H^{0,1}\qquad\quad\text{by type of form}.\label{DecompositionFour}
\end{align}

Given a Hermitian metric $H$ on a Higgs bundle $(E, \bar\partial_E, \phi$), we can construct a new connection~$D$ on~$E$ as
\begin{gather}\label{PlusThree}
D=\nabla_{\bar\partial_E, H}+\phi+\phi^{*_H},
\end{gather}
where $\nabla_{\bar\partial_E, H}$ is the Chern connection determined by $\bar\partial_E$ and $H$.

\begin{Lemma}\label{NameHarmonic}
A harmonic metric $H$ on a flat bundle $(E, D)$ over $\Sigma$ implies that the triple $\big(E, D_H^{0,1}, \Psi_H^{1,0}\big)$ obtained in equation~\eqref{DecompositionFour} is a Higgs bundle. Conversely, given a Higgs bundle $(E, \bar\partial_{E}, \phi)$ together with a Hermitian metric~$H$ such that the new connection $D=\nabla_{\bar\partial_E, H}+\phi+\phi^{*_H}$ in equation~\eqref{PlusThree} is flat, then the metric~$H$ is harmonic on the flat bundle $(E, D)$.
\end{Lemma}
\begin{proof} The proof is purely algebraic.
\begin{enumerate}\itemsep=0pt
\item[(1)] $D_H(*\Psi_H)=0$ (harmonicity) and
\item[(2)] $F_D=0$ (flatness) implies
\begin{enumerate}\itemsep=0pt
\item[(2a)] $F_{D_H}+\Psi_H\wedge\Psi_H=0$ and
\item[(2b)] $D_H\Psi_H=0$.
\end{enumerate}
\end{enumerate}
One can check that (1) and (2b) together imply
\begin{enumerate}\itemsep=0pt
\item[(3)] $(D_H)^{0,1}\Psi_H^{1,0}=0$.
\end{enumerate}
Note that (3) and the choice of $D_H$ imply that $D_H$ is the Chern connection determined by the holomorphic structure~$D_H^{0,1}$ and the Hermitian metric~$H$.

Using the same algebra calculation, one can show that the converse is also true.
\end{proof}

Therefore, we may rephrase the theorem of Corlette ad Donaldson as follows.

\begin{Theorem}\label{FlatConnectionHiggsBundle}
Given $D$ a flat irreducible ${\rm SL}(n,\mathbb C)$-connection on a vector bundle $E$ over $\Sigma$ satisfying $\det E=O$, there exists a unique $($up to a scalar multiple$)$ Hermitian metric $H$ such that $\big(E, D_H^{0,1}, \Psi_H^{1,0}\big)$ is a ${\rm SL}(n,\mathbb C)$-Higgs bundle on~$\Sigma$.
\end{Theorem}
\begin{Definition}\quad
\begin{enumerate}\itemsep=0pt
\item[(1)] A Higgs bundle $(E,\phi)$ of degree $0$ is stable if for every proper $\phi$-invariant holomorphic subbundle $F$ has a negative degree.
\item[(2)] A ${\rm SL}(n, \mathbb C)$-Higgs bundle $(E,\phi)$ is polystable if it is a direct sum of stable Higgs bundles of degree~$0$.
\end{enumerate}
\end{Definition}
The Higgs bundle obtained in Theorem~\ref{FlatConnectionHiggsBundle} is polystable, see the proof in Lemma~\ref{NonCyclic}.

\subsubsection{From flat bundles to harmonic maps}
Therefore, we may rephrase the theorem of Corlette and Donaldson as follows.

\begin{Theorem}\label{FlatConnectionHiggsBundle2}
Given $D$ a flat irreducible ${\rm SL}(n,\mathbb C)$-connection on a vector bundle $E$ over $\Sigma$ with its holonomy ``representation" $\rho\colon \pi_1(S)\rightarrow {\rm SL}(n, \mathbb C)$, there exists a unique $\rho$-equivariant harmonic map $f\colon \widetilde\Sigma\rightarrow {\rm SL}(n,\mathbb C)/{\rm SU}(n)$.
\end{Theorem}

\subsubsection {From Higgs bundles to flat bundles}
\begin{Theorem}[Hitchin \cite{Hitchin87}, Simpson \cite{Simpson88}]
Let $(E, \phi)$ be a polystable ${\rm SL}(n, \mathbb C)$-Higgs bundle, then there exists a Hermitian metric $H$ on $E$ whose induced Hermitian metric $\det H$ on \mbox{$\det E\cong \mathcal O$} is~$1$, and such that
\begin{gather*}D=\nabla_{\bar\partial_E, H}+\phi+\phi^{*_H}\end{gather*} is flat, where $\nabla_{\bar\partial_E, H}$ is the Chern connection uniquely determined by $H$ and $\bar\partial_E$, and $\phi^{*_H}$ is the Hermitian adjoint of $\phi$.

If $(E,\phi)$ is stable, the metric is unique.
\end{Theorem}

\begin{Remark}From Lemma \ref{NameHarmonic}, we can see that a harmonic metric on a Higgs bundle is indeed a harmonic metric on the associated flat bundle.
\end{Remark}
The connection $D$ here is a reductive ${\rm SL}(n,\mathbb C)$-connection and one may refer to~\cite{Wentworth} for the proof.

We note that the connection $D=\nabla_{\bar\partial_E, H}+\phi+\phi^{*_H}$ being flat is equivalent to the \textit{Hitchin equation}
\begin{gather}\label{HitchinEquation}
F_{\nabla_{\bar\partial_E, H}}+[\phi,\phi^{*_H}]=0,\end{gather}
where $F_{\nabla_{\bar\partial_E, H}}$ is the curvature of the Chern connection $\nabla_{\bar\partial_E, H}$ and the Lie bracket $[\phi,\phi^{*_H}]$ is defined as follows:

First take $\phi\wedge \phi^{*_H}\in\Omega^{1,1}(\Sigma, \operatorname{End}(E)\otimes \operatorname{End}(E))$ and then apply the generalized Lie bracket on the tensor product $\operatorname{End}(E)\otimes \operatorname{End}(E)$. Using this definition, one can check \begin{gather}\label{LieBracket}
[\phi,\phi^{*_H}]=\phi\wedge\phi^{*_H}+\phi^{*_H}\wedge\phi,\end{gather}
where the $\wedge$ operator in equation (\ref{LieBracket}) means doing wedge product on forms and composition on sections of $\operatorname{End}(E)$ at the same time, different from the one in the beginning of this paragraph.

\begin{Definition}\quad
\begin{enumerate}\itemsep=0pt
\item[(1)] The space of gauge equivalence classes of polystable ${\rm SL}(n,\mathbb C)$-Higgs bundles is called the moduli space of ${\rm SL}(n,\mathbb C)$-Higgs bundles and we denote it by $\mathcal{M}_{\rm Higgs}({\rm SL}(n,\mathbb C))$.
\item[(2)] The space of gauge equivalence classes of reductive flat ${\rm SL}(n,\mathbb C)$-connections is called the de Rham moduli space and we denote it by $\mathcal{M}_{\rm deRham}({\rm SL}(n,\mathbb C))$.
\item[(3)] The space of conjugacy classes of reductive representations from $\pi_1(S)$ into ${\rm SL}(n,\mathbb C)$ is called the representation variety and we denote it by $\operatorname{Rep}(\pi_1S, {\rm SL}(n,\mathbb C))$.
\item[(4)] The space of equivariant harmonic maps from $\widetilde\Sigma$ to $N$ modulo isometries in $N$ is denoted by $\mathcal {H}$.
\end{enumerate}
\end{Definition}

From the discussion in this section, we obtain a 1-1 correspondence \begin{align*}
\operatorname{NAH}_{\Sigma}\colon \ &\mathcal{M}_{\rm Higgs}({\rm SL}(n,\mathbb C))\cong \mathcal{H}\cong \mathcal{M}_{\rm deRham}({\rm SL}(n,\mathbb C))\cong \operatorname{Rep}(\pi_1(S), {\rm SL}(n,\mathbb C))\\
& (E,\phi)\longmapsto \big(f\colon \widetilde\Sigma\rightarrow N\big)\longmapsto D\longmapsto\text{the holonomy of $D$}.
\end{align*} This is called the non-abelian Hodge correspondence.

\begin{Remark} One can generalize the non-abelian Hodge correspondence to general real reductive Lie groups, see \cite{GGM}. For a general Lie group $G$, we consider equivariant harmonic maps from~$\widetilde\Sigma$ to the symmetric space $G/K$, where $K$ is a maximal subgroup of $G$ (unique up to conjugacy). In later sections, we'll directly mention $G$-Higgs bundles without more explanation.
\end{Remark}

\begin{Remark} If a reductive representation $\rho$ of $\pi_1(S)$ into ${\rm SL}(n,\mathbb C)$ has image inside a proper subgroup~$G$ of ${\rm SL}(n,\mathbb C)$, the corresponding $\rho$-equivariant harmonic map will lie in the totally geodesic submanifold~$G/K$ inside~$N$ where~$K$ is a maximal compact subgroup of~$G$.
\end{Remark}

\section[Several concepts in $\mathcal M_{\rm Higgs}({\rm SL}(n,\mathbb C))$]{Several concepts in $\boldsymbol{\mathcal M_{\rm Higgs}({\rm SL}(n,\mathbb C))}$}\label{Two}
In this section, we introduce several important concepts for $\mathcal M_{\rm Higgs}({\rm SL}(n,\mathbb C))$: the Hitchin fibration, the Hitchin section, the $\mathbb C^*$-action, the Morse function, cyclic Higgs bundles and a~discussion of stability.

\subsection{Hitchin fibration}
Given a basis of ${\rm SL}(n,\mathbb C)$-invariant homogeneous polynomials $p_i$ of degree $i$ on $\mathfrak{sl}(n,\mathbb C)$, $2\leq i\leq n$, the Hitchin fibration is a map from the moduli space of ${\rm SL}(n,\mathbb{C})$-Higgs bundles over $\Sigma$ to the direct sum of holomorphic differentials
\begin{align*}
h\colon \ M_{\rm Higgs}({\rm SL}(n,\mathbb C)) &\longrightarrow \bigoplus\limits_{j=2}^nH^0\big(\Sigma, K^j\big),\\
(E,\phi)&\longmapsto (p_2(\phi),\dots,p_n(\phi)).\end{align*}
We call each fiber of the Hitchin fibration a Hitchin fiber. The Hitchin fiber over the origin is called the nilpotent cone.
\begin{Remark}\label{HitchinfibrationHopf}Note that $p_2(\phi)$ is always a constant multiple of $\operatorname{tr}\big(\phi^2\big)$. Hence the first term of the image $h(E,\phi)$ of the Hitchin fibration coincides with the Hopf differential of the associated harmonic map $f\colon \widetilde\Sigma\rightarrow N$ up to a scalar multiple, see Section~\ref{HopfDifferentialSection}.
\end{Remark}

\subsection{Hitchin section}
By choosing an appropriate basis of polynomials $p_i$'s, the Hitchin section $s$ of the Hitchin fibration can be defined explicitly as follows. Denote by $K^{\frac{1}{2}}$ a holomorphic line bundle such that its square is the canonical line bundle $K$. Define \begin{gather}
s(q_2,q_3,\dots,q_n)\nonumber\\
\qquad {}=\left(E=K^{\frac{n-1}{2}}\oplus K^{\frac{n-3}{2}}\oplus\cdots\oplus K^{\frac{1-n}{2}},\,\phi=\begin{pmatrix}
0&q_2&q_3&\cdots&q_n\\
r_1&0&q_2&\cdots&q_{n-1}\\
&r_2&0&\ddots&\vdots\\
&&\ddots&\ddots&q_2\\
&&&r_{n-1}&0\end{pmatrix}\right),\label{HitchinSection}\end{gather} where $r_i=\frac{i(n-i)}{2}$ for $1\leq i\leq n-1$.

Hitchin in \cite{Hitchin92} showed that the Higgs bundles in the image of Hitchin section have holonomy in ${\rm SL}(n,\mathbb R)$. Moreover, the corresponding representations form a connected component of the ${\rm SL}(n,\mathbb R)$-representation variety, called the Hitchin component and denoted by $\operatorname{Hit}_n$. The Hitchin component also descends to a connected component in the ${\rm PSL}(n,\mathbb R)$-representation variety and is also called the Hitchin component. Labourie in \cite{LabourieAnosov} showed that Hitchin representations are Anosov and hence they are discrete, faithful quasi-isometric embeddings of $\pi_1(S)$ into ${\rm PSL}(n,\mathbb R)$.

When $n=2$, the Higgs bundles in the Hitchin section form exactly the Higgs bundle parametrization of the \textit{Teichm\"uller space}, the space of isotopy classes of hyperbolic metrics on the surface $S$. We will see more details on this in Section~\ref{ExampleRank2}. The corresponding representations of $\pi_1(S)$ into ${\rm PSL}(2,\mathbb R)$ are Fuchsian, that is, discrete and faithful.

The image $s(q_2,0,\dots,0)$ corresponds to an embedding of the Teichm\"uller space inside the Hitchin section. Each representation corresponding to $s(q_2,0,\dots,0)$ for some $q_2$ is a Fuchsian representation post-composing with the unique irreducible representation from ${\rm PSL}(2,\mathbb R)$ to ${\rm PSL}(n,\mathbb R)$, called \textit{an $n$-Fuchsian representation}.

\begin{Remark}One can also define Hitchin representations for split real Lie groups, see Hit\-chin~\cite{Hitchin92}.
\end{Remark}

\subsection{Maximal representations}
For a reductive representation $\rho$ into ${\rm Sp}(2n,\mathbb R)$, we can define the Toledo invariant \begin{gather*}\tau(\rho):=\frac{2}{\pi}\int_Sf^*\omega,\end{gather*} where $f\colon \widetilde{S}\rightarrow {\rm Sp}(2n,\mathbb R)/{\rm U}(n)$ is any $\rho$-equivariant continuous map and $\omega$ is the normalized ${\rm Sp}(2n,\mathbb R)$-invariant K\"ahler $2$-form on ${\rm Sp}(2n,\mathbb R)/{\rm U}(n)$. The Toledo invariant satisfies the Milnor--Wood inequality $|\tau(\rho)|\leq n(g-1)$ shown in \cite{Burger1}. A representation $\rho$ with $|\tau(\rho)|=n(g-1)$ is called maximal. Corresponding to the representations of~$\pi_1(S)$ into ${\rm Sp}(2n,\mathbb R)$, the Higgs bundles over $\Sigma$ are of the form $\left(V\oplus V^*,\left(\begin{smallmatrix}0&\beta\\ \gamma&0\end{smallmatrix}\right)\right)$ where $V$ is a rank $n$ holomorphic vector bundle over~$\Sigma$, $\beta\in H^0\big(S^2V\otimes K_{\Sigma}\big)$ and $\gamma\in H^0\big(\Sigma,S^2V^*\otimes K_{\Sigma}\big)$. The integer $\deg V$ is the Toledo invariant of the representation and hence for a maximal ${\rm Sp}(2n,\mathbb R)$-representation, the corresponding Higgs bundle has~$|\deg V|=n(g-1)$.

Maximal representations are Anosov \cite{Burger} and hence they are discrete, faithful quasi-isometric embeddings of $\pi_1(S)$ into ${\rm Sp}(2n,\mathbb R)$.

For ${\rm Sp}(4,\mathbb R)$, there are $3\cdot 2^{2g}+2g-4$ connected components of maximal representations containing~$2^{2g}$ Hitchin components~\cite{Hitchin92} and $2g-3$ exceptional components called Gothen components~\cite{Gothen}. With the description in~\cite{BradlowDeformation, Gothen}, any maximal representation into ${\rm Sp}(4,\mathbb{R})$ in the Gothen components and the Hitchin components corresponds to a Higgs bundle of the form
\begin{gather*}
E=N\oplus NK^{-1}\oplus N^{-1}K\oplus N^{-1}, \qquad \phi=
\left(
\begin{matrix}
0 &q_2 &0& \nu\\
1 &0 & 0&0\\
0&\mu &0 &q_2 \\
0&0 & 1 & 0
\end{matrix}
\right),
\end{gather*} where $N$ is a holomorphic line bundle over $\Sigma$ satisfying $g-1<\deg N\leq 3g-3$, $q_2\in H^0\big(\Sigma, K^2\big)$, $\mu\in H^0\big(\Sigma, N^{-2}K^3\big)$, and $\nu\in H^0\big(\Sigma, N^2K\big)$.
Note that if $N=K^{\frac{3}{2}}$, the above Higgs bundle lies in the Hitchin section.

\begin{Remark}One can also consider maximal representations in general Hermitian type groups, see \cite{Burger1}.
\end{Remark}

\subsection[The $\mathbb C^*$-action]{The $\boldsymbol{\mathbb C^*}$-action}
There is a natural $\mathbb{C}^*$-action on the moduli space of ${\rm SL}(n,\mathbb{C})$-Higgs bundles:
\begin{align*}
\mathbb{C}^*\times \mathcal{M}_{\rm Higgs}({\rm SL}(n,\mathbb C))&\longrightarrow \mathcal{M}_{\rm Higgs}({\rm SL}(n,\mathbb C)),\\
t\cdot [(E,\phi)]&= [(E,t\phi)].\end{align*}

The $\mathbb C^*$-action takes the Hitchin fiber at $(q_2,\dots, q_n)$ to the Hitchin fiber at $\big(t^2q_2,\dots, t^nq_n\big)$. So the $\mathbb C^*$-action always takes Higgs bundles in a Hitchin fiber to another distinct Hitchin fiber unless the Higgs bundles are in the nilpotent cone.

\subsection{The Morse function}\label{Morse}
We can define a nonnegative function $f\colon \mathcal{M}_{\rm Higgs}({\rm SL}(n,\mathbb C))\rightarrow \mathbb R$ by
\begin{gather*}
f([E,\phi])=\int_\Sigma ||\phi||^2\operatorname{dvol}_{g_0}=i\int_\Sigma \operatorname{tr}(\phi\wedge\phi^*).
\end{gather*}
In fact, $f$ is a Morse-Bott function on the smooth locus of $\mathcal M_{\rm Higgs}({\rm SL}(n,\mathbb C))$. Moreover, the critical points of $f$ are exactly the fixed points of the $\mathbb C^*$-action on the moduli space. Hitchin in~\cite{Hitchin87} showed that the function is proper, which makes it an important tool to study the topology of the moduli space.

\subsection{Cyclic Higgs bundles}\label{Three}
\begin{Definition}A cyclic Higgs bundle $(E,\phi)$ over $\Sigma$ is a ${\rm SL}(n,\mathbb C)$-Higgs bundle of the form
\begin{eqnarray}\label{CyclicForm}E=L_1\oplus L_2\oplus\cdots\oplus L_n,\qquad\phi=\begin{pmatrix}&&&&\gamma_n\\\gamma_1&&&&\\&\gamma_2&&&\\&&\ddots&&\\
&&&\gamma_{n-1}&\end{pmatrix},\end{eqnarray}
where $L_i$'s are holomorphic line bundles, $\gamma_i\in H^0\big(\Sigma,L_i^{-1}L_{i+1}K\big)$ and for $1\leq i\leq n-1$, $\gamma_i\neq 0$.

We call a cyclic Higgs bundle real if $L_i=L_{n+1-i}^{-1}$ for $1\leq i\leq n$ and $\gamma_i=\gamma_{n-i}$ for $1\leq i\leq n-1$.
\end{Definition}

We note that cyclic Higgs bundles always lie in the Hitchin fiber at $(0,\dots,0, q_n)$. The following lemma is the main reason why cyclic Higgs bundles are particularly interesting.
\begin{Lemma}[Baraglia \cite{Bar}]\label{diagonal} Stable cyclic Higgs bundles have diagonal harmonic metrics.
\end{Lemma}
\begin{proof}Consider the gauge transformation $g=\operatorname{diag}\big(1,\omega,\dots,\omega^{n-1}\big)$, where $\omega={\rm e}^{\frac{2\pi{\rm i}}{n}}$. Since
\begin{gather*}
g\phi g^{-1}=\begin{pmatrix}&&&&\omega^{1-n}\gamma_n\\
\omega\gamma_1&&&&\\
&\omega\gamma_2&&&\\
&&\ddots&&\\
&&&\omega\gamma_{n-1}&\end{pmatrix}=\omega\cdot \phi,
\end{gather*} we have $g\cdot (E,\phi)=(E,\omega\cdot \phi)$.
If $H$ solves the Hitchin equation of $(E,\phi)$, then $g\cdot H=\big(\bar g^t\big)^{-1} Hg^{-1}$ solves the Hitchin equation of $g\cdot (E,\phi)=(E,\omega\phi)$. We can see that $H$ also solves the Hitchin equation of $(E,\omega\phi)$. By the uniqueness of a harmonic metric, $g\cdot H=H$ and $\big(\bar g^t\big)^{-1} Hg^{-1}=H$. Hence~$H$ is diagonal.
\end{proof}

Using the uniqueness of a harmonic metric and a similar method in Lemma~\ref{diagonal}, one can show the following result and we leave it as an exercise.
\begin{Exercise}\label{CyclicReal}
Stable cyclic real Higgs bundles have diagonal harmonic metrics $H=(h_1,\dots, h_n)$ satisfying $h_i=h_{n+1-i}^{-1}$ for $1\leq i\leq n$.
\end{Exercise}

\begin{Example}\quad
\begin{enumerate}\itemsep=0pt
\item Inside $\mathcal M_{\rm Higgs}({\rm SL}(n,\mathbb C))$, every Higgs bundle in the image $s(0,\dots,0,q_n)$ of the Hitchin section is cyclic.
\item Inside $\mathcal M_{\rm Higgs}({\rm SL}(2,\mathbb C))$, every Higgs bundle in the Hitchin section is cyclic.
\item Inside $\mathcal M_{\rm Higgs}({\rm SL}(4,\mathbb C))$, every Higgs bundle in the Hitchin fiber at $(0,0,q_4)$ which corresponds to a representation in a Gothen component for ${\rm Sp}(4,\mathbb R)$ is cyclic.
\end{enumerate}
\end{Example}

 \subsection{Stability}
We give some examples of stable Higgs bundles.
\begin{Proposition}\label{Stable}\quad
\begin{enumerate}\itemsep=0pt
\item[$(1)$] Every Higgs bundle in the Hitchin section is stable.
\item[$(2)$] Every Higgs bundle in the Gothen component for ${\rm Sp}(4,\mathbb R)$ is stable.
\item[$(3)$] If a cyclic Higgs bundle $(E,\phi)$ of the form \eqref{CyclicForm} satisfies $\sum\limits_{i=1}^k\deg L_{n+1-i}<0$ for each $1\leq k\leq n-1$, then $(E,\phi)$ is stable.
\end{enumerate}
\end{Proposition}

\textit{Sketch of the proof}. The followings two facts about stability: (i) The $\mathbb C^*$-action preserves stability; (ii) Stability is an open condition, prove that if $\lim\limits_{t\rightarrow 0}t\cdot [(E,\phi)]$ is stable, then $(E,\phi)$ is stable. Proposition~\ref{Stable} follows from directly checking the stability of $\lim\limits_{t\rightarrow 0}t\cdot [(E,\phi)]$.

\part{{Analysis and geometry on the Hitchin equation}}\label{Part2}
\section{Local expression of the Hitchin equation}
Consider a local coordinate chart $U$ of $\Sigma$ where the bundle $E$ has a local holomorphic trivia\-li\-zation over $U$ by choosing a local holomorphic frame $e=(e_1,e_2,\dots,e_n)$. We denote by $h$ the matrix presentation whose $(i,j)$-entry $h_{ij}$ is the pairing $H(e_i, e_j)$. For any two local sections $s=e\cdot \xi$, $t=e\cdot \eta$ of~$E$ over~$U$, where $\xi,\eta\in \Omega^0(U, \mathbb C^n)$, the pairing of $s$, $t$ is given by
\begin{gather}\label{MetricLocal}
H(s, t)={\bar\xi}^t\cdot h\cdot\eta.\end{gather}

We are going to write the Hitchin equation (\ref{HitchinEquation}) in terms of the local frame $e$. Let's first write the curvature $F_{\nabla_{\bar\partial_E, H}}$ and the term $[\phi,\phi^{*_H}]$ as follows:

$\bullet$ \textit{Curvature $F_{\nabla_{\bar\partial_E, H}}$}:

The Chern connection $\nabla_{\bar\partial_E, H}$ is
\begin{gather*}\nabla_{\bar\partial_E, H}=\nabla_{\bar\partial_E, H}^{1,0}+\nabla_{\bar\partial_E, H}^{0,1}=\nabla_{\bar\partial_E, H}^{1,0}+\bar\partial_E.\end{gather*} Locally the holomorphic structure $\bar\partial_E$ on $E$ is just $\bar\partial$. Let's first write the operator $\nabla_{\bar\partial_E, H}^{1,0}$ in local expression. Assume that $\nabla_{\bar\partial_E, H}^{1,0}e=e\cdot A$, for some $A\in \Omega^{1,0}(U,\operatorname{End}(\mathbb C^n))$ and the local expression of $\nabla_{\bar\partial_E, H}^{1,0}$ is $\nabla_{\bar\partial_E, H}^{1,0}=\partial+A$.
Recall equation~(\ref{HolomorphicChern}) as follows
\begin{gather}\label{HolomorphicChern1}\partial(H(s,t))=H(\bar\partial_E s,t)+H\big(s,\nabla_{\bar\partial_E, H}^{1,0} t\big).\end{gather}
Using equation (\ref{MetricLocal}) and the assumption that the frame $e$ is holomorphic, equation (\ref{HolomorphicChern1}) becomes
\begin{gather*}
\partial\big(\bar\xi^t\cdot h\cdot \eta\big)=H\big(e\cdot\bar\partial\xi, e\cdot\eta\big)+H\big(e\cdot\xi,\nabla_{\bar\partial_E, H}^{1,0}e\cdot \eta+e\cdot\partial\eta\big)\\
\qquad{} \Longrightarrow\quad\partial\bar\xi^t\cdot h\cdot \eta+\bar\xi^t\cdot \partial h\cdot \eta+\bar\xi^t\cdot h\cdot \partial \eta=\overline{\bar\partial\xi}^t\cdot h\cdot \eta+\bar\xi^t\cdot h\cdot A\eta+\bar\xi^t\cdot h\cdot\partial\eta.
\end{gather*}
This implies that $A=h^{-1}\partial h$. Therefore $\nabla_{\bar\partial_E, H}={\rm d}+A={\rm d}+h^{-1}\partial h$ and thus the curvature~$F_{\nabla_{\bar\partial_E, H}}$ is given by
\begin{gather}\label{CurvatureChern}F_{\nabla_{\bar\partial_E, H}}=\nabla_{\bar\partial_E, H}\circ \nabla_{\bar\partial_E, H}=({\rm d}+A)\circ({\rm d}+A)={\rm d}A+A\wedge A
=\bar\partial\big(h^{-1}\partial h\big).\end{gather}
In the case that $E$ is a line bundle and $h$ is a local function, the curvature $F_{\nabla_{\bar\partial_E, H}}$ is locally $\bar\partial\partial\log h$.

$\bullet$ \textit{The term $[\phi,\phi^{*_H}]$}:

For a local section $s=e\cdot \xi$ of $E$, set $\hat\phi, \hat\phi^{*_H}\in \Omega^0(U,\operatorname{End}(\mathbb C^n))$ such that \begin{gather*}\phi(s)=e\cdot\hat\phi \xi\cdot {\rm d}z,\qquad \phi^{*_H}(s)=e \cdot \hat\phi^{*_H}\xi\cdot {\rm d}\bar z.\end{gather*} Using the formula (\ref{LieBracket}) of $[\phi,\phi^{*_H}]$, the term $[\phi,\phi^{*_H}]$ is given by \begin{gather}\label{LieBracketLocal}
[\phi,\phi^{*_H}]=\big[\hat\phi, \hat\phi^{*_H}\big]{\rm d}z\wedge {\rm d}\bar z,\end{gather}
where the Lie bracket on the right hand is the usual Lie bracket for matrices.
The only remaining term to understand is $\hat\phi^{*_H}$. First, by definition, $\phi^{*_H}$ is such that $H(\phi(s),t)=H(s,\phi^{*_H}(t))$. Therefore we have
\begin{gather*}
H\big(e\cdot\hat\phi\xi\cdot {\rm d}z,e\cdot\eta\big)=H\big(e\cdot \xi,e\cdot \hat\phi^{*_H}\eta\cdot {\rm d}\bar z\big)\\
\qquad{} \Longrightarrow \quad \overline{\hat\phi\xi}^t\cdot h\cdot \eta=\bar\xi^t\cdot h\cdot \hat\phi^{*_H}\cdot\eta\quad \Longrightarrow \quad \hat\phi^{*_H}=h^{-1}\bar{\hat\phi}^th.
\end{gather*}

$\bullet$ \textit{The Hitchin equation}:

Combining equations (\ref{CurvatureChern}) and (\ref{LieBracketLocal}), the Hitchin equation (\ref{HitchinEquation}) is locally
\begin{gather}\label{LocalHitchinEquation}\bar\partial(h^{-1}\partial h)+\big[\hat\phi, \hat\phi^{*_H}\big]{\rm d}z\wedge {\rm d}\bar z=0,\end{gather}
where $ \hat\phi^{*_H}=h^{-1}\bar{\hat\phi}^th$.

\section{Harmonic maps in terms of Higgs bundles}
Let the Riemann surface $\Sigma$ be equipped with a background conformal metric $g_0$. Suppose we are given a polystable ${\rm SL}(n,\mathbb C)$-Higgs bundle $(E,\phi)$ over $\Sigma$ together with a harmonic metric~$H$, then we obtain a flat ${\rm SL}(n,\mathbb C)$-connection $D=\nabla_{\bar\partial_E, H}+\phi+\phi^{*_H}$ with its holonomy as $\rho\colon \pi_1(S)\rightarrow {\rm SL}(n,\mathbb C)$. Meanwhile, we obtain a $\rho$-equivariant harmonic map $f\colon \big(\widetilde S, \widetilde g_0\big)\rightarrow N\cong {\rm SL}(n,\mathbb C)/{\rm SU}(n)$. Now we discuss in this section the data of the harmonic map $f$ in terms of $(E,\phi, H)$ consisting of tangent vector, energy density, energy, Hopf differential and curvature.
\subsection{Tangent vector}
Using Lemma \ref{TangentHiggs}, $-\frac{1}{2}f^{-1}{\rm d}f=\Psi_H$. By decomposing into $(1,0)$- and $(0,1)$-forms, we have \begin{gather*}\Psi_H=\phi+\phi^{*_H}.\end{gather*}
Therefore $f^{-1}\partial f=-2\phi$, and $f^{-1}\bar\partial f=-2\phi^{*_H}$.

\subsection{Energy density and energy} Following from the formula~(\ref{EnergyDensity}), the energy density of $f$ is given by
\begin{gather*}e(f)=\tfrac{1}{2}\langle {\rm d}f,{\rm d}f\rangle =n\langle \Psi_H,\Psi_H\rangle =n\cdot \operatorname{tr}(\Psi_H\wedge\star\Psi_H)/\omega=2{\rm i}n\cdot\operatorname{tr}(\phi\wedge\phi^{*_H})/\omega,\end{gather*}
where we use that $\star\Psi_H=\star(\phi+\phi^{*_H})={\rm i}(\phi^{*_H}-\phi)$.

Following from the formula (\ref{Energy}), the energy of $f$ is given by \begin{gather*}
E(f)=\int_{S} e(f)\operatorname{dvol}_{g_0}=2{\rm i}n\int_{\Sigma}\operatorname{tr}(\phi\wedge\phi^{*_H}).\end{gather*} This is also the Morse function on the moduli space of Higgs bundles in Section~\ref{Morse}.

\subsection{Pullback metric} Following from the formula (\ref{Metric}), the pullback metric $f^*g_N$ is given by
\begin{align}
f^*g_N&= 2n\operatorname{tr}\big(\phi^2\big)+e(f)\cdot g_0+2n\overline{\operatorname{tr}\big(\phi^2\big)}\nonumber\\
&= 2n\cdot\big(\operatorname{tr}\big(\phi^2\big){\rm d}z^2+\operatorname{tr}(\phi\phi^{*_H})({\rm d}z\otimes {\rm d}\bar z+{\rm d}\bar z\otimes {\rm d}z)+\overline{\operatorname{tr}\big(\phi^2\big)}\big({\rm d}\bar z^2\big)\big).\label{PullbackMetric}\end{align}
\begin{Remark}\label{Immerse}
Note that the pullback metric is only a semi-positive symmetric $2$-tensor. From the above expression of the pullback metric, the pullback metric degenerates at $p$ when \begin{gather*}\operatorname{tr}(\phi\phi^{*_H})^2-\big|\operatorname{tr}\big(\phi^2\big)\big|^2=0,\qquad\text{at $p$}\end{gather*} equivalently, when \begin{gather*}\phi^*(p)=\lambda\cdot \phi(p),\qquad\text{for some $\lambda\in {\rm U}(1)$}. \end{gather*} Only when the map $f$ is an immersion, the pullback metric $f^*g_N$ is indeed a metric.
\end{Remark}
\begin{Remark}If $f$ is conformal and hence minimal, then $f^*g_N=2n\cdot\operatorname{tr}(\phi\phi^{*_H})({\rm d}z\otimes {\rm d}\bar z+{\rm d}\bar z\otimes {\rm d}z)$. Therefore~$f$ is a minimal immersion if and only if~$\phi$ does not vanish anywhere. In particular, an equivariant minimal mapping for Hitchin representations is automatically an immersion.
\end{Remark}

\subsection{Hopf Differential}\label{HopfDifferentialSection}The Hopf differential of a smooth map $f\colon \Sigma\rightarrow N$ is defined to be the $(2,0)$-part of the pullback metric $f^*g_N$, denoted by $\operatorname{Hopf}(f)$. A~map~$f$ is conformal if and only if $\operatorname{Hopf}(f)=0$. If a~map~$f$ is harmonic, then its Hopf differential $\operatorname{Hopf}(f)$ is holomorphic. From equation~(\ref{PullbackMetric}), the Hopf differential of the harmonic map $f$ is given by
\begin{gather*}
\operatorname{Hopf}(f)=({f^*g_N})^{2,0}=2n\cdot\operatorname{tr}\big(\phi^2\big),\end{gather*} which is a holomorphic quadratic differential over $\widetilde\Sigma$ and descends to $\Sigma$.

\subsection[Curvature of the pullback metric $f^*g_N$]{Curvature of the pullback metric $\boldsymbol{f^*g_N}$}

Denote by $\kappa$ the Gaussian curvature of the pullback metric $f^*g_N$ on~$\Sigma$. For a tangent plane $\sigma\subset T_{f(x)}N$ at $f(x)$ which is tangential to $f\big(\widetilde\Sigma\big)$, denote by $k_{\sigma}^N$ the sectional curvature of~$\sigma$ in~$N$. We then have the following proposition.
\begin{Proposition}\label{Reducing}
At every immersed point $x\in \widetilde\Sigma$, the following holds:
\begin{gather*}
\kappa\leq k^N_{\sigma}=-\frac{1}{2n}\cdot\frac{\operatorname{tr}([\phi,\phi^{*_H}]^2)}{\operatorname{tr}(\phi\phi^{*_H})^2-\big|\operatorname{tr}\big(\phi^2\big)\big|^2}\leq 0.
\end{gather*}
Moreover, the equality of the first inequality holds at $x$ if and only if the map $f$ is totally geodesic at~$x$.
\end{Proposition}

The first inequality is proven in \cite[Lemma~C.4]{Reznikov}, \cite[Theorem~7]{Sampson}, and reproven in \cite[Lemma~2.5]{DominationFuchsian}. We include an argument here for its importance.

\begin{proof}For the first inequality:

Let $U\subset \widetilde\Sigma$ be a domain containing $x$ for $f$ being immersed everywhere. Let $e_1$, $e_2$ be an orthonormal basis of the induced metric at $f(x)\in N$. The Gauss formula for the curvature is
\begin{gather}\label{GaussFormula}
\kappa=k_{\sigma}^N+\langle II(e_1,e_1),II(e_2,e_2)\rangle -|II(e_1,e_2)|^2,
\end{gather}
where $II$ is the second fundamental form for the embedded image $f(U)$ defined by $II(X,Y)=(\nabla_XY)^{\perp}$ with respect to the tangent plane of~$f(U)$ and the Levi-Civita connection~$\nabla$ on~$N$.

Let $\sigma_1$, $\sigma_2$ be an orthonormal basis of $g_0$ at $x\in U$. By Definition~\ref{Harmonic}, the harmonicity of $f$ means
\begin{gather}\label{harmonic}\operatorname{tr}_{g_0}\nabla {\rm d}f=\nabla {\rm d}f(\sigma_1,\sigma_1)+\nabla {\rm d}f(\sigma_2,\sigma_2)=0,\end{gather} where $\nabla {\rm d}f(X,Y)=\nabla_X ({\rm d}f(Y))-{\rm d}f(\nabla_XY)$ is the second fundamental form of the map $f$. By projection to the normal bundle, equation (\ref{harmonic}) implies
\begin{gather}\label{harmonic1}
II({\rm d}f(\sigma_1),{\rm d}f(\sigma_1))+II({\rm d}f(\sigma_2),{\rm d}f(\sigma_2))=0.\end{gather}
Since $f$ is immersed at $x$, we have
\begin{eqnarray*}
e_1=a\cdot {\rm d}f(\sigma_1)+b\cdot {\rm d}f(\sigma_2),\qquad e_2=c\cdot {\rm d}f(\sigma_1)+{\rm d}\cdot {\rm d}f(\sigma_2).
\end{eqnarray*} where $ad-bc\neq 0$.
Using equation~(\ref{harmonic1}) and the symmetry of $II$, denote \begin{gather*}
x=II({\rm d}f(\sigma_1),{\rm d}f(\sigma_1))=-II({\rm d}f(\sigma_2),{\rm d}f(\sigma_2)),\\
 y=II({\rm d}f(\sigma_1),{\rm d}f(\sigma_2))=II({\rm d}f(\sigma_2),{\rm d}f(\sigma_1)).\end{gather*}So equation~(\ref{GaussFormula}) becomes
\begin{align*}
\kappa&= k_{\sigma}^N+\langle II(e_1,e_1),II(e_2,e_2)\rangle -|II(e_1,e_2)|^2\\
&= k_{\sigma}^N+\big\langle \big(a^2-b^2\big)x+2aby,\big(c^2-d^2\big)x+2cdy\big\rangle -|(ac-bd)x+(bc+ad)y|^2\\
&= k_{\sigma}^N-(ad-bc)^2\big(|x|^2+|y|^2\big)\leq k_{\sigma}^N.
\end{align*}
Equality holds if and only if $x=y=0$ since $ad-bc\neq 0$.

For the second equality:

At an immersed point $p$, the sectional curvature $k^N_{\sigma}$ of the tangent plane $\sigma$ at $f(p)\in N$ which is tangential to $f\big(\widetilde \Sigma\big)$ is given by: suppose $f(p)=\big(g^{-1}\big)^*g^{-1}$ for some $g\in {\rm SL}(n,\mathbb C)$,
\begin{align}\label{CurvatureForm}
k_{\sigma}^N&= K_{f(p)}(f_x,f_y)\\
&= K\big(A{\rm d}\big(g^{-1}\big)\big({-}\tfrac{1}{2}f(p)^{-1}f_x\big),A{\rm d}\big(g^{-1}\big)\big({-}\tfrac{1}{2}f(p)^{-1}f_y\big)\big)\label{equation1}\\
&= K\big({-}\tfrac{1}{2}f(p)^{-1}f_x, -\tfrac{1}{2}f(p)^{-1}f_y\big)\label{equation2}\\
&= K(\Psi_H(\partial_x), \Psi_H(\partial_y))\label{equation3}\\
&= \frac{1}{2n}\cdot\frac{\operatorname{tr}\big(\big[\hat\phi+\hat\phi^{*_H},{\rm i}\big(\hat\phi-\hat\phi^{*_H}\big)\big]^2\big)}{\operatorname{tr}\big(\hat\phi+\hat\phi^{*_H}\big)^2\cdot \operatorname{tr}\big(\hat\phi+\hat\phi^{*_H}\big)^2-\big(\operatorname{tr}\big(\hat\phi+\hat\phi^{*_H}\big)\big({\rm i}\big(\hat\phi-\hat\phi^{*_H}\big)\big)\big)^2}\label{equation4}\\
&= -\frac{1}{2n}\cdot\frac{\operatorname{tr}\big(\big[\hat\phi,\hat\phi^{*_H}\big]^2\big)}{\operatorname{tr}\big(\hat\phi\hat\phi^{*_H}\big)^2-\big|\operatorname{tr}\big(\hat\phi^2\big)\big|^2}.\label{equation5}
\end{align}
Here, equation (\ref{equation1}) follows from the formula (\ref{PullbackVector}); equation (\ref{equation3}) follows from the curvature formula (\ref{CurvatureOnp}) is invariant under adjoint action; equation~(\ref{equation2}) follows from Lemma~\ref{TangentHiggs}; equation~(\ref{equation4}) follows from $\Psi_H=\phi+\phi^{*_H}$, $B(X,Y)=2n\cdot \operatorname{tr}(XY)$ and the curvature formula (\ref{CurvatureOnp}); and equation (\ref{equation5}) is a direct calculation.

We obtain the last inequality by using that the sectional curvature of ${\rm SL}(n,\mathbb C)/{\rm SU}(n)$ is nonpositive which follows from the fact that the Killing form $B$ is negative definite on $su(n)$.
\end{proof}

\begin{Remark}
If $f$ is in addition conformal and hence minimal, then away from the zeros of $\phi$, the sectional curvature $k_{\sigma}^N$ is given by \begin{gather*}k_{\sigma}^N=-\frac{1}{2n}\cdot\frac{\operatorname{tr}\big([\phi,\phi^{*_H}]^2\big)}{\operatorname{tr}(\phi\phi^{*_H})^2}\leq 0.\end{gather*}
\end{Remark}

\section{Examples of different rank}

From now on, we choose $g_0$ to be the Hermitian hyperbolic metric on $\Sigma$. The metric $g_0$ is locally given by
\begin{gather*}g_0=g_0(z)({\rm d}z\otimes {\rm d}\bar z+{\rm d}\bar z\otimes {\rm d}z)=2g_0(z)\big({\rm d}x^2+{\rm d}y^2\big).\end{gather*} Since the local Gaussian curvature formula of $g_0$ is $K_{g_0}=-\frac{1}{g_0(z)}\partial_{\bar z}\partial_z\log g_0(z)$ and the Gaussian curvature of $g_0$ is $-1$, the local function $g_0(z)$ satisfies \begin{gather}\label{HyperbolicFormula}\partial_{\bar z}\partial_z\log g_0(z)=g_0(z).\end{gather} Note that the Hermitian metric $g_0$ on $\Sigma$ induces a Hermitian metric on $K_{\Sigma}^{-1}$, also denoted as $g_0$.

\subsection{Rank 2}\label{ExampleRank2}
Consider the Higgs bundle $\left(E=K^{\frac{1}{2}}\oplus K^{-\frac{1}{2}}, \phi=\left(\begin{smallmatrix}0&q_2\\1&0\end{smallmatrix}\right)\right)$, where $q_2\in H^0\big(\Sigma, K^2\big)$. From Proposition~\ref{Stable}, $(E,\phi)$ is stable. Clearly $(E,\phi)$ is cyclic and thus the harmonic metric splits on the holomorphic decomposition $E=K^{\frac{1}{2}}\oplus K^{-\frac{1}{2}}$ following from Lemma~\ref{diagonal}. Choose a local coordinate $z$ and a local holomorphic frame $e=(e_1,e_2)$, where $e_1={\rm d}z^{\frac{1}{2}}$ and $e_2={\rm d}z^{-\frac{1}{2}}$. Since $\phi(e_1)=e_2\cdot {\rm d}z$ and $\phi(e_2)=e_1\cdot q_2(z){\rm d}z$, we have
\begin{gather*}
\phi(e\cdot\begin{pmatrix}1\\0\end{pmatrix})=e\cdot\begin{pmatrix}0\\1\end{pmatrix}\cdot {\rm d}z, \qquad \phi(e\cdot\begin{pmatrix}0\\1\end{pmatrix})=e\cdot\begin{pmatrix}q_2(z)\\0\end{pmatrix}\cdot {\rm d}z.
\end{gather*}
So $\hat\phi=\left(\begin{smallmatrix}0&q_2(z)\\1&0\end{smallmatrix}\right)$. Let $h_1=H(e_1,e_1)$, $h_2=H(e_2,e_2)$ and hence $h=\operatorname{diag}(h_1, h_2)$. From $\det h=1$, we have $h_2=h_1^{-1}$. By direct calculations, we have
\begin{gather}
\bar\partial\big(h^{-1}\partial h\big)=\begin{pmatrix}\partial_{\bar z}\partial_z\log h_1&0\\0&-\partial_{\bar z}\partial_z\log h_1\end{pmatrix}\cdot {\rm d}\bar z\wedge {\rm d}z,\nonumber\\
\label{LieBracketRank2}
\big[\hat\phi, \hat\phi^{*_H}\big]=\begin{pmatrix}|q_2(z)|^2h_1^2-h_1^{-2}&0\\0&h_1^{-2}-|q_2(z)|^2h_1^2\end{pmatrix}.
\end{gather}
From equation (\ref{LocalHitchinEquation}), the Hitchin equation (\ref{HitchinEquation}) reduces to a single scalar equation
\begin{gather}\label{SingleHitchin}
\partial_{\bar z}\partial_z\log h_1+h_1^{-2}-|q_2(z)|^2h_1^2=0.
\end{gather}

Let $h_1=g_0(z)^{-\frac{1}{2}}{\rm e}^u$, where $u$ is a smooth function over $\Sigma$. Combining equations~(\ref{SingleHitchin}) and~(\ref{HyperbolicFormula}), we have
\begin{gather}\label{LocalSingleHitchin}
\partial_{\bar z}\partial_z u+g_0(z){\rm e}^{-2u}-|q_2(z)|^2{\rm e}^{2u}\cdot g_0(z)^{-1}-\tfrac{1}{2}g_0(z)=0.
\end{gather}
The operator $\frac{1}{g_0(z)}\partial_{\bar z}\partial_z$ and the function $|q_2(z)|^2g_0(z)^{-2}$ do not depend on the choice of coordina\-te~$z$, and are denoted by $\triangle_{g_0}$ and $||q_2||_{g_0}^2$ respectively. Hence equation~(\ref{LocalSingleHitchin}) becomes
\begin{gather}\label{HitchinScalar2}
\triangle_{g_0} u+{\rm e}^{-2u}-||q_2||_{g_0}^2{\rm e}^{2u}-\tfrac{1}{2}=0,
\end{gather}
which indeed holds globally.

\begin{Claim}\label{ClaimRank2} Suppose $q_2\neq 0$, then
(i) ${\rm e}^{-2u}> \frac{1}{2}$, (ii) $||q_2||_{g_0}^2{\rm e}^{4u}<1$.
\end{Claim}
Note that for the case $q_2=0$, ${\rm e}^{-2u}\equiv \frac{1}{2}$.
\begin{proof}
(i)~Estimate equation (\ref{HitchinScalar2}), at maximum of $u$, we have $\triangle_{g_0} u\leq 0$ and hence ${\rm e}^{-2u}-||q_2||_{g_0}^2{\rm e}^{2u}-\frac{1}{2}\geq 0$ and then ${\rm e}^{-2u}\geq \frac{1}{2}$. By the strong maximum principle (see \cite{MaximumPrinciple}), ${\rm e}^{-2u}>\frac{1}{2}$.
(ii)~Combining equations (\ref{HitchinScalar2}) and~(\ref{HyperbolicFormula}), we obtain the following equation of $||q_2||_{g_0}^2{\rm e}^{4u}$: away from zeros of $q_2$,
\begin{gather*}
\triangle_{g_0}\log \big(||q_2||_{g_0}^2{\rm e}^{4u}\big)=4{\rm e}^{-2u}\big(||q_2||_{g_0}^2{\rm e}^{4u}-1\big).
\end{gather*}
At maximum of $||q_2||_{g_0}^2{\rm e}^{4u}$, we have $||q_2||_{g_0}^2{\rm e}^{4u}-1\leq 0$. By the strong maximum principle, $||q_2||_{g_0}^2{\rm e}^{4u}-1<0$.
\end{proof}

Part (i) of Claim \ref{ClaimRank2} implies that the energy density $e(f)=4\big({\rm e}^{-2u}+||q_2||_{g_0}^2{\rm e}^{2u}\big)\geq 4{\rm e}^{-2u}>2$; and Part (ii) of Claim~\ref{ClaimRank2} implies that $[\phi,\phi^*]\neq 0$ using equation~(\ref{LieBracketRank2}).

\textbf{Geometric interpretation.} From the formula (\ref{PullbackMetric}), the pullback metric is given by \begin{gather*}f^*g_{{\rm SL}(2,\mathbb R)/{\rm SO}(2)}=8q_2{\rm d} z^2+4\big({\rm e}^{-2u}+||q_2||_{g_0}^2{\rm e}^{2u}\big)g_0+8\bar q_2{\rm d}\bar z^2.\end{gather*} From Remark~\ref{Immerse}. Part (ii) of Claim~\ref{ClaimRank2} says the corresponding equivariant harmonic map $f\colon \widetilde\Sigma\rightarrow {\rm SL}(2,\mathbb R)/{\rm SO}(2)$ is an immersion and hence is a diffeomorphism. It is shown by Hitchin that the pullback metric $f^*g_{{\rm SL}(2,\mathbb R)/{\rm SO}(2)}$ being of constant curvature $-\frac{1}{2}$ is equivalent to equation~(\ref{HitchinScalar2}) (Theorem~11.2 in Hitchin \cite{Hitchin87}). In fact, the Hitchin section in $\mathcal M_{\rm Higgs}({\rm SL}(2,\mathbb C))$ provides a parametrization of the Teichm\"uller space.

One may compare the above Higgs bundles parametrization with Wolf's parametrization of the Teichm\"uller space by the space of holomorphic quadratic differentials using harmonic maps in~\cite{TeichOfHarmonic}. For a hyperbolic metric~$h$ on~$S$, by Eells--Sampson~\cite{EellsSampson}, there exists a unique harmonic map~$f_h$ from $\Sigma$ to~$(S,h)$ isotopic to identity. In this way, we associate a holomorphic quadratic differential $\operatorname{Hopf}(f_h)$ to~$h$. Conversely, for a holomorphic quadratic differential $q_2$, one can find a hyperbolic metric~$h$ on~$S$ such that the unique harmonic map $f\colon \Sigma\rightarrow (S, h)$ isotopic to the identity has~$q_2$ as its Hopf differential. Using the harmonicity equation for $f$ to deduce a Bochner-type equation for the term $-\frac{1}{2}\log||\partial f||^2$, where $||\partial f||^2$ is the holomorphic energy density of~$f$, one can recover equation (\ref{HitchinScalar2}), see Schoen--Yau \cite{SchoenYauUnivalent}. So Part~(i) of Claim~\ref{ClaimRank2} implies that the holomorphic energy density is greater than~$2$.

\subsection{Rank 3}\label{ExampleRank3}
Consider the Higgs bundle $\left(E=K\oplus \mathcal O\oplus K^{-1}, \phi=\left(\begin{smallmatrix}0&0&q_3\\1&0&0\\0&1&0\end{smallmatrix}\right)\right)$, where $q_3\in H^0\big(\Sigma, K^3\big)$. Choose a local coordinate~$z$ and a local holomorphic frame $e=(e_1,e_2, e_3)$, where $e_1={\rm d}z$, $e_2=1$ and $e_2={\rm d}z^{-1}$. Hence $h=\operatorname{diag}\big(h_1, 1,h_1^{-1}\big)$ from Exercise~\ref{CyclicReal} and Proposition~\ref{Stable}. By direct calculations, the Hitchin equation~(\ref{HitchinEquation}) reduces to a single scalar equation
\begin{gather}\label{SingleHitchin3}
\partial_{\bar z}\partial_z\log h_1+h_1^{-1}-|q_3(z)|^2h_1^2=0.
\end{gather}
Let $h_1=g_0(z)^{-1}{\rm e}^u$, where $u$ is a smooth function over $\Sigma$. As in Example~\ref{ExampleRank2}, from equations~(\ref{SingleHitchin3}) and~(\ref{HyperbolicFormula}) it follows that
\begin{gather}\label{HitchinScalar3}
\triangle_{g_0} u+{\rm e}^{-u}-||q_3||_{g_0}^2{\rm e}^{2u}-1=0,
\end{gather}
where $||q_3||_{g_0}^2=|q_3(z)|^3g_0(z)^{-3}$.

Adapting the estimates in Example \ref{ExampleRank2}, we have the following claim.
\begin{Claim}\label{ClaimRank3} Suppose $q_3\neq 0$, then
(i) ${\rm e}^{-u}>1$, and (ii) $||q_3||_{g_0}^2{\rm e}^{3u}<1$.
\end{Claim}
Note that for the case $q_3=0$, ${\rm e}^{-u}\equiv 1$. Part~(i) of Claim~\ref{ClaimRank3} implies that the energy density $e(f)=6\cdot \big(2{\rm e}^{-u}+||q_3||_{g_0}^2{\rm e}^{2u}\big)\geq 12{\rm e}^{-u}>12$; and Part (ii) of Claim~\ref{ClaimRank3} implies that $[\phi,\phi^*]\neq 0$.

\textbf{Geometric interpretation}:

The corresponding equivariant harmonic map $f\colon \widetilde\Sigma\rightarrow {\rm SL}(3,\mathbb R)/{\rm SO}(3)$ is in fact minimal since its Hopf differential is $2n\cdot \operatorname{tr}\big(\phi^2\big)=0$. The minimal surface is closely related to the hyperbolic affine sphere in affine geometry. Let's roughly explain their relation as follows. One may check \cite{DumasWolf, LabourieFlat, Loftin, LoftinSurvey} for more details.

Consider a locally strictly convex hypersurface $M\subset\mathbb{R}^3$ where~$\mathbb R^3$ is equipped with a volume form. Affine differential geometry associates to such a locally convex surface a special transverse vector field, the affine normal. Being a hyperbolic affine sphere means that all the affine normals of each image point meet at a~point, which lies in a convex side of the hypersurface. Relative to the affine normal, there are two objects on~$M$: (1) the second fundamental form induces a Riemannian metric $h$ with conformal coordinate~$z$, called the \textit{Blaschke metric}; (2) a cubic differential $q_3{\rm d}z^3$, called the \textit{Pick differential}, measuring the difference between the induced connection of the hypersurface and the Levi-Civita connection with respect to~$h$. The Blaschke lift $f$ is a map from $M$ to the space $\operatorname{Met}\big(\mathbb R^3\big)$ of Euclidean metrics on $E$ of volume 1:
 $f\colon M\ni s\rightarrow f(s)\in \operatorname{Met}\big(\mathbb R^3\big)\cong {\rm SL}(3,\mathbb R)/{\rm SO}(3)$
 such that $f(s)(X,\lambda)=g_{s}(X,X)+\lambda^2$, where $\mathbb R^3$ is identified with $T_{s}M\oplus \mathbb R\cdot s$ and $g_{s}$ is the Blaschke metric on~$T_{s}M$. Then~$M$ is a hyperbolic affine sphere if and only if its Blaschke lift $f\colon M\rightarrow {\rm SL}(3,\mathbb R)/{\rm SO}(3)$ is a minimal surface.

For a representation $\rho\colon \pi_1(S)\rightarrow {\rm SL}(3,\mathbb R)$, we consider a hypersurface~$M\subset\mathbb R^3$ invariant under the action of the subgroup $\rho(\pi_1(S))$ of ${\rm SL}(3,\mathbb R)$. We can reparametrize the hypersurface~$M$ by a $\rho$-equivariant map $\iota\colon \widetilde\Sigma\rightarrow \mathbb{R}^3$ such that the induced Blaschke metric is conformal. Following Wang~\cite{Wang} and Simon--Wang \cite{SimonWang}, $M$ being the affine sphere is equivalent to the Pick differential~$q_3$ being holomorphic and the Blaschke metric $h=2g_0\cdot {\rm e}^{-u}$ where $(h, q_3)$ satisfies \textit{Wang's equation}, which is equivalent to the pair $(u, 2q_3)$ satisfying equation~(\ref{HitchinScalar3}). Note that both the Blaschke metric and the Pick differential descend to the Riemann surface $\Sigma$.

Part (i) in Claim \ref{ClaimRank3} says the Blaschke metric strictly dominates the conformal hyperbolic metric pointwise. Part~(ii) of Claim~\ref{ClaimRank3} says the curvature of the Blaschke metric is strictly negative.

\subsection[Rank $n$]{Rank $\boldsymbol{n}$}\label{ExampleRankN}
For stable cyclic Higgs bundles of the form (\ref{CyclicForm}), the harmonic metric is \begin{gather*}H=\operatorname{diag}(h_1,h_2,\dots,h_n)\end{gather*} from Lemma~\ref{diagonal}. Choose a local coordinate~$z$ and a local holomorphic frame~$e$. Let $h_i$, $\gamma_i$ also denote the Hermitian metric on each~$L_i$ and the holomorphic map~$\gamma_i$ with respect to the frame~$e$ respectively. So the Hitchin equation is locally
\begin{gather}
\partial_{\bar z}\partial_z\log h_1-\big(h_1h_n^{-1}|\gamma_n|^2-h_1^{-1}h_2|\gamma_1|^2\big)=0,\nonumber\\
\partial_{\bar z}\partial_z\log h_2-\big(h_1^{-1}h_2|\gamma_1|^2-h_2^{-1}h_3|\gamma_2|^2\big)=0,\nonumber\\
\cdots\cdots\cdots\cdots\cdots\cdots\cdots\cdots\cdots\cdots\cdots\cdots\cdots\cdots\cdots\nonumber\\
\partial_{\bar z}\partial_z\log h_n-\big(h_{n-1}^{-1}h_n|\gamma_{n-1}|^2-h_1h_n^{-1}|\gamma_n|^2\big)=0.\label{CyclicEquation}
\end{gather}

The above Hitchin equation coincides with the ``2-dimensional Toda equation with opposite sign'', which is a classical object in integrable systems, for example, see~\cite{Guest, LinGuest}. Baraglia in~\cite{Bar1} first introduced the notion of cyclic Higgs bundles and related them to the Toda equations.

Dai and Li in \cite{DaiLi2} studied the harmonic maps associated to cyclic Higgs bundle. There, the main tool for studying cyclic Higgs bundles is the following maximum principle for systems.
\begin{Lemma}[Dai--Li \cite{DaiLi2}]\label{MaximumPrinciple}
Let $(\Sigma,g)$ be a closed Riemannian surface. For each $1\leq i\leq n$, let~$u_i$ be a~$C^2$ function on $\Sigma{\setminus} P_i$, where~$P_i$ is an isolated subset of $\Sigma$ $(P_i$ can be empty$)$. Suppose~$u_i$ approaches $+\infty$ around~$P_i$. Let $P=\bigcup_{i=1}^{n} P_i$ and~$c_{ij}$ be bounded continuous functions on $\Sigma{\setminus} P$, $1\leq i,j\leq n$.
Suppose $c_{ij}$ satisfy the following assumptions: in $\Sigma{\setminus} P$,
\begin{enumerate}\itemsep=0pt
\item[$(a)$] cooperative: $c_{ij}\geq 0$, $i\neq j$,
\item[$(b)$] column diagonally dominant: $\sum\limits_{i=1}^{n}c_{ij}\leq 0$, $1\leq j\leq n$,
\item[$(c)$] fully coupled: the index set $\{1,\dots,n\}$ cannot be split up in two disjoint nonempty sets~$\alpha$,~$\beta$ such that $c_{ij}\equiv 0$ for $i\in\alpha$, $j\in \beta$.
\end{enumerate}
Let $f_i$ be non-positive continuous functions on $\Sigma{\setminus} P$, $1\leq i\leq n$ on $\Sigma$. Suppose $u_i$'s satisfy
\begin{gather*}
\triangle_{g} u_i+\sum\limits_{j=1}^{n}c_{ij}u_j=f_i\qquad \text{in} \quad \Sigma{\setminus} P, \quad 1\leq i\leq n.
\end{gather*}
Consider the following conditions:
\begin{enumerate}\itemsep=0pt
\item[$(1)$] $(f_1,\dots,f_n)\!\neq\! (0{,}\dots{,}0)$, i.e., there exists $i_{0}\!\in\! \{1{,}\dots{,} n\}$ and $p_0\!\in\! \Sigma{\setminus} P$ such that \mbox{$f_{i_0}(p_0)\!\neq\!0$};
\item[$(2)$] $P$ is nonempty;
\item[$(3)$] $\sum\limits_{i=1}^n u_i\geq 0$.
\end{enumerate}
Then either condition $(1)$ or $(2)$ imply that $u_{i}> 0$, $1\leq i\leq n$. And condition $(3)$ implies that either $u_{i}> 0$, $1\leq i\leq n$ or $u_{i}\equiv 0$, $1\leq i\leq n$.
\end{Lemma}

Here we give a quick application of the maximum principle Lemma~\ref{MaximumPrinciple} to the case of cyclic Higgs bundles in the Hitchin section as
\begin{gather*} E=K^{\frac{n-1}{2}}\oplus K^{\frac{n-3}{2}}\oplus\cdots\oplus K^{\frac{1-n}{2}},\qquad\phi=\begin{pmatrix}
0&&&&q_n\\
r_1&0&&&\\
&r_2&0&&\\
&&\ddots&\ddots&\\
&&&r_{n-1}&0\end{pmatrix},\end{gather*} where $r_i=\frac{i(n-i)}{2}\colon K^{\frac{n+1-2i}{2}}\rightarrow K^{\frac{n-1-2i}{2}}\otimes K$ for $1\leq i\leq n-1$.
 Suppose $n=2m$, the solution to the Hitchin equation is $h=\operatorname{diag}\big(h_1,\dots,h_m,h_m^{-1},\dots,h_1^{-1}\big)$, by Exercise~\ref{CyclicReal} and Proposition~\ref{Stable}.

\begin{Claim}[Dai--Li \cite{DaiLi2}]\label{ClaimRankN}
Suppose $q_n\neq 0$, then $||q_n||^2<||r_1||^2<||r_2||^2<\cdots<||r_m||^2$ where $||q_n||^2$, $||r_i||^2$ use the pairing using formula~\eqref{GeneralizedPairing}.
\end{Claim}
\begin{proof} Let $u_0=\log \big(||q_n||^2\big)=\log \big(|q_n|^2h_1^2/g_0\big)$ and $u_k=\log\big(||r_k||^2\big)=\log \big(r_k^2h^{-1}_kh_{k+1}/g_0\big)$ for $1\leq k\leq m$. Using equation~(\ref{CyclicEquation}), the equations of $u_k$'s are given by: away from zeros of $q_n$,
\begin{gather*}
\triangle_{g_0} u_0+2{\rm e}^{u_1}-2{\rm e}^{u_0} = 0,\\
\triangle_{g_0} u_k+{\rm e}^{u_{k+1}}-2{\rm e}^{u_k}+{\rm e}^{u_{k-1}} = 0, \qquad k=1,\dots, m-1,\\
\triangle_{g_0} u_m-2{\rm e}^{u_m}+2{\rm e}^{u_{m-1}} = 0.
\end{gather*}

Let $v_k=u_{k+1}-u_k$ for $0\leq k\leq m-1$. By linearization, set $c_k=\int_0^1{\rm e}^{tu_{k+1}+(1-t)u_k}{\rm d}t$. The equations of~$v_k$'s are given by: away from the zeros of~$q_n$,
\begin{gather*}
\triangle_{g_0} v_0-3c_0v_0+c_1v_1 = 0, \\
\triangle_{g_0} v_k+c_{k-1}v_{k-1}-2c_kv_k+c_{k+1}v_{k+1} = 0, \qquad k=1,\dots, m-2,\\
\triangle_{g_0} v_{m-1}+c_{m-2}v_{m-2}-3c_{m-1}v_{m-1} = 0.
\end{gather*}
Then by checking the conditions, we can apply the maximum principle Lemma~\ref{MaximumPrinciple} and obtain $v_k>0$ for $0\leq k\leq m-1$.

Therefore for each $1\leq k\leq m$, $u_k>u_{k-1}$ and the claim follows.
\end{proof}

For more applications of the maximum principle Lemma~\ref{MaximumPrinciple}, one can refer to~\cite{DaiLi2}. An immediate corollary of Claim~\ref{ClaimRankN} is $[\phi,\phi^{*_H}]\neq 0$.

For higher rank cyclic Higgs bundles, there is not much geometry behind other than considering the harmonic map from $\widetilde\Sigma$ to the symmetric space except for the case $n=4$. In the case $n=4$, there is a related geometric object, maximal space-like surfaces in~$\mathbb H^{2,2}$, developed by Collier--Tholozan--Toulisse~\cite{CollierTholozanToulisse}, which can be viewed as a generalization of the hyperbolic affine sphere for the case $n=3$.

\subsection[Rank $n$ but not cyclic]{Rank $\boldsymbol{n}$ but not cyclic} \label{NonCyclic}
For a given Higgs bundle in the Hitchin section of the form~(\ref{HitchinSection}), the harmonic metric is in general not diagonal. In this case, we aim to get a system of elliptic inequalities from the Hitchin equation. Before that, let's first review the proof showing that the existence of a harmonic metric implies polystability.

\begin{Lemma}The existence of a harmonic metric on a ${\rm SL}(n,\mathbb C)$-Higgs bundle $(E,\phi)$ implies the Higgs bundle $(E,\phi)$ is polystable.
\end{Lemma}
\begin{proof}
To show the polystability, we only need to consider holomorphic $\phi$-invariant subbundles. But we start from general holomorphic subbundles for later use.

For a holomorphic subbundle $F$ of $E$, we would like to deduce the Hitchin equation which respects $F$. One can check more details on this calculation in~\cite{LiHitchin}. Denote by $F^{\perp}$ the subbundle of $E$ perpendicular to $F$ with respect to the harmonic metric~$H$. $F^{\perp}$ can be equipped with the quotient holomorphic structure from $E/F$. With respect to the $C^{\infty}$ orthogonal decomposition \begin{gather*}E=F\oplus F^{\perp},\end{gather*} we have the expression of the holomorphic structure $\bar\partial_E$ and the Higgs field $\phi$ as follows:
\begin{gather*}
\bar\partial_E=\begin{pmatrix}\bar\partial_F&\beta\\
0&\bar\partial_{F^{\perp}}\end{pmatrix},\qquad \phi=\begin{pmatrix}\phi_1&\alpha\\B&\phi_2\end{pmatrix},\qquad H=\begin{pmatrix}H_1&0\\0&H_2\end{pmatrix},
\end{gather*}
where $B\in \Omega^{1,0}\big(\Sigma, \operatorname{Hom}\big(F, F^{\perp}\big)\big)$, $\alpha\in \Omega^{1,0}\big(\Sigma, \operatorname{Hom}\big(F^{\perp},F\big)\big)$, and $\beta\in \Omega^{0,1}\big(\Sigma, \operatorname{Hom}\big(F^{\perp},F\big)\big)$.

The Chern connection $\nabla_{\bar\partial_E, H}$ and the adjoint $\phi^{*_H}$ of the Higgs field are
\begin{gather*}
\nabla_{\bar\partial_E, H}=\begin{pmatrix}
\nabla_{\bar\partial_F, H_1}&\beta\\
-\beta^{*_H}&\nabla_{\bar\partial_{F^{\perp}}, H_2}\end{pmatrix},\qquad
\phi^{*_H}=\begin{pmatrix}
\phi_1^{*_{H_1}}&B^{*_H}\\
\alpha^{*_H}&\phi_2^{*_{H_2}}\end{pmatrix}.
\end{gather*}

We calculate the Hitchin equation with respect to the decomposition $E=F\oplus F^{\perp}$ and by restricting to $\operatorname{Hom}(F, F)$, we obtain
\begin{gather*}
F_{\nabla_{\bar\partial_F, H_1}}-\beta\wedge\beta^{*_H}+\alpha\wedge\alpha^{*_H}+B^{*_H}\wedge B+\big[\phi_1,\phi_1^{*_{H_1}}\big]=0.
\end{gather*}
By taking trace and noting that $\operatorname{tr}\big(\big[\phi_1,\phi_1^{*_{H_1}}\big]\big)=0$, we obtain
\begin{gather*}
\operatorname{tr}\big(F_{\nabla_{\bar\partial_F, H_1}}\big)-\operatorname{tr}(\beta\wedge\beta^{*_H})+\operatorname{tr}(\alpha\wedge\alpha^{*_H})+\operatorname{tr}(B^{*_H}\wedge B)=0.
\end{gather*}

Using the formula of pairings in equation (\ref{GeneralizedPairing}), we obtain a scalar equation on the surface
\begin{gather*}\operatorname{tr}\big(F_{\nabla_{\bar\partial_F, H_1}}\big)-{\rm i}\big(||\beta||^2+||\alpha||^2-||B||^2\big)\omega=0.
\end{gather*}
Therefore,
\begin{gather}\label{standard}
{\rm i}\Lambda \operatorname{tr}\big(F_{\nabla_{\bar\partial_F, H_1}}\big)-||B||^2\leq 0.
\end{gather}

When $F$ is $\phi$-invariant, that is, $B$ vanishes everywhere, this says that $\Lambda \operatorname{tr}\big(F_{\nabla_{\bar\partial_F, H_1}}\big)\leq 0$ and hence
\begin{gather*}
\deg F=\frac{{\rm i}}{2\pi}\int_\Sigma \Lambda \operatorname{tr}\big(F_{\nabla_{\bar\partial_F, H_1}}\big) \omega \leq 0.
\end{gather*}

Moreover, if the equality holds, then both $\beta$, $\alpha$ vanish everywhere. In this case, we obtain $(E, \phi)$ is the direct sum of two Higgs bundles $(F, \phi_1)$ and $\big(F^{\perp}, \phi_2\big)$, both of which are of degree~$0$.

Therefore, the existence of a harmonic metric implies that the Higgs bundle is polystable.
\end{proof}

For the case when $(E,\phi)$ is a Higgs bundle in the Hitchin section of the form (\ref{HitchinSection}), we are going to show the following claim.
\begin{Claim}(Li \cite{LiHitchin})
If $(E,\phi)$ is not in the nilpotent cone, then the energy density satisfies $e(f)>\frac{n^4-n^2}{6}$. \end{Claim}
Note that if $(E,\phi)$ is in the nilpotent cone, then the energy density satisfies $e(f)\equiv\frac{n^4-n^2}{6}$ and we leave it as an exercise.
\begin{proof}

We can choose the holomophic subbundle $F_k$ of $E$ to be the direct sum of the first $k$-holomorphic line bundles, $F_k=\bigoplus\limits_{i=1}^k K^{\frac{n+1-2i}{2}}$. We will denote by $H_k$ the induced metric on $F_k$ from the harmonic metric $H$ on $E$. The bundle $E$ admits a holomorphic filtration
\begin{gather*} 0=F_0\subset F_1\subset \cdots \subset F_n=E.\end{gather*}
And for each $k$, the Higgs field $\phi$ takes $F_k$ to $F_{k+1}\otimes K$ and the induced map of $\phi\colon F_k/F_{k-1}\cong K^{\frac{n+1-2k}{2}}\rightarrow (F_{k+1}/F_k)\otimes K\cong K^{\frac{n-1-2k}{2}}\otimes K$ is the constant map $r_k=\frac{k(n-k)}{2}$.

The quotient line bundles $F_k/F_{k-1}$ are equipped with the Hermitian metric $H_k/H_{k-1}$. Using the formula of pairings in equation~(\ref{GeneralizedPairing}), the square of~$r_k$\rq{}s norm is given by \begin{gather*}
 ||r_k||^2=r_k\wedge r_k^*/\omega=r_k^2\cdot (\det H_k/\det H_{k-1})^{-1}\cdot (\det H_{k+1}/\det H_k).
 \end{gather*}
Let $B_k\in \Omega^{1,0}\big(\Sigma, \operatorname{Hom}\big(F_k, F_k^{\perp}\big)\big)$ denote the induced map of the Higgs field which induces the map~$r_k$. We leave the following statement as an exercise.

 \begin{Exercise}
$||B_k||^2=||r_k||^2$.
 \end{Exercise}

Note that $\operatorname{tr}\big(F_{\bar\partial_{F_k}, \nabla_{H_k}}\big)$ is exactly the curvature on the determinant line bundle~$\det F_k$, i.e., $F_{\nabla_{\bar\partial_{\det F_k}, \det H_k}}$. Applying equation (\ref{standard}) to each subbundle $F_k$, we obtain a system of elliptic inequalities:
\begin{gather*}
\Lambda F_{\nabla_{\bar\partial_{\det F_1}, \det H_1}}-||r_1||^2 \leq 0,\\
\Lambda F_{\nabla_{\bar\partial_{\det F_2}, \det H_2}}-||r_2||^2 \leq 0,\\
\cdots\cdots\cdots\cdots\cdots\cdots\cdots\cdots\cdots\cdots\cdots \\
\Lambda F_{\nabla_{\bar\partial_{\det F_{n-1}},\det H_{n-1}}}-||r_{n-1}||^2 \leq 0.
\end{gather*}
Locally $\Lambda F_{\nabla_{\bar\partial_{\det F_k}, \det H_k}}=\bar\partial\partial\log (\det H_k)$.
We can either apply Lemma~\ref{MaximumPrinciple} or directly argue and obtain that for each $k$, $\det H_k/h^{\frac{k(n-k)}{2}}<1$.

This is the main step for the proof. The interested reader may refer to~\cite{LiHitchin} for the rest of the proof.
\end{proof}

\part[Selected topics on harmonic maps and minimal surfaces]{Selected topics on harmonic maps\\ and minimal surfaces}\label{Part3}
\section{Labourie's conjecture}

For a fixed Riemann surface $\Sigma$, the Hitchin component $\operatorname{Hit}_n$ is parametrized by $\bigoplus\limits_{i=2}^n H^0\big(\Sigma,K^i\big)$. Denote by $\mathcal V$ the vector bundle over the Teichm\"uller space whose fiber at $\Sigma$ is the vector space $\bigoplus\limits_{i=3}^nH^0\big(\Sigma, K^i\big)$ and Labourie in~\cite{LabourieEnergy} considered the map
\begin{align*}
\mathcal V&\longrightarrow \operatorname{Hit}_n\subset \operatorname{Rep}(\pi_1S, {\rm PSL}(n,\mathbb R)),\\
(\Sigma, (q_3,\dots,q_n))&\longmapsto \operatorname{NAH}_{\Sigma}(0,q_3,\dots,q_n).
\end{align*}
The left hand side has the same dimension as the right hand side. The map is equivariant with respect to the mapping class group action. In the same paper, Labourie asked the following question.
\begin{Question}\label{Question} Is this map a bijection?\end{Question}
As noted in Remark \ref{HitchinfibrationHopf}, the vanishing of the first term in the image $h(E,\phi)$ in the Hitchin fibration is equivalent to the vanishing of the Hopf differential of the associated harmonic map. Also, if the Hopf differential vanishes, the harmonic map is conformal and thus minimal. Therefore, Question \ref{Question} can be generalized and rephrased as follows:
\begin{Conjecture}[Labourie\rq{}s conjecture] For $\rho$ a Hitchin representation into a split real Lie group $G$ or a maximal representation into a Hermitian Lie group~$G$, does there exist a unique $\rho$-equivariant minimal surface in $G/K$?
\end{Conjecture}

$\bullet$ Existence (shown by Labourie \cite{LabourieEnergy}.) Labourie in \cite{LabourieEnergy} showed that for a fixed Anosov representation $\rho$, the function $f_{\rho}\colon T(S)\rightarrow \mathbb R$ sending each $\Sigma$ to $E\big(\operatorname{NAH}_{\Sigma}^{-1}(\rho)\big)$ is proper. So the function $f_{\rho}$ has a critical point. By the classical results of Sacks--Uhlenbeck~\cite{SacksUhlenbeck1, SacksUhlenbeck2} and Schoen--Yau~\cite{SchoenYau}, the critical Riemann surface~$\Sigma$ is such that the corresponding harmonic map is conformal. That is, $\operatorname{tr}\big(\phi^2\big)=0$. Since both Hitchin representations and maximal representations are Anosov, the existence follows.

$\bullet$ Uniqueness. This is the main part of Labourie's conjecture.

Labourie's conjecture is proven for Hitchin representations into ${\rm SL}(2,\mathbb R)\times {\rm SL}(2,\mathbb R)$, by Schoen in \cite{Schoen}; ${\rm SL}(3, \mathbb R)$, independently by Labourie \cite{LabourieFlat}, Loftin \cite{Loftin} and all the remaning rank 2 split real Lie groups ${\rm SL}(3, \mathbb R)$, ${\rm Sp}(4,\mathbb R)$, $G_2$ by Labourie~\cite{LabourieCyclic}. The property for cyclic Higgs bundles in Lemma \ref{diagonal} is essential in Labourie's proof~\cite{LabourieCyclic}.

Labourie's conjecture is proven for maximal representations into ${\rm Sp}(4,\mathbb R)$ by Collier \cite{Collier}; ${\rm PSp}(4,\mathbb R)$ by Alessandrini--Collier \cite{AlessandriniCollier} and all the remaining rank 2 Hermitian Lie groups by Collier--Tholozan--Toulisse \cite{CollierTholozanToulisse}.

\section{Asymptotics}
Before we introduce the asymptotic question, let us first recall Thurston's compactifica\-tion \mbox{\cite{FLP, ThurstonCompactification}} of the Teichm\"uller space with measured foliations. Let~$\mathcal S$ denote the space of isotopy classes of simple closed curves and denote the projectivization of the space of nonnegative functions on~$\mathcal S$ by~${\mathbb{PR}}^\mathcal S_+$. The map which assigns the projectivized length spectrum of each hyperbolic metric is an embedding of the Teichm\"uller space inside ${\mathbb{PR}}^\mathcal S_{+}$. The boundary corresponds to the image of the intersection length spectrum of the space of projective measured foliations. In terms of the representation variety, there are also algebraic techniques on the compactification by Morgan--Shalen \cite{MorganShalenCompact} and generalized by Parreau in~\cite{Parreau} for the representation variety for higher rank Lie groups.

Wolf in~\cite{TeichOfHarmonic} recovers Thurston's compactification using harmonic maps. Fix a Riemann surface structure~$\Sigma$, the Teichm\"uller space is homeomorphic to the vector space $H^0\big(\Sigma, K^2\big)$, see Section~\ref{ExampleRank2}. Roughly, the harmonic map compactification of the Teichm\"uller space is by adding the space of rays in the vector space. Let $q_2$ be a holomorphic quadratic differential, consider the ray $tq_2$ and let $h_t$ be the corresponding family of hyperbolic metrics such that the associated harmonic maps $f_t\colon \Sigma \rightarrow (S, h_t)$ have Hopf differential $tq_2$. Away from the zeros of $q_2$ choose a~coordinate $z$ such that $q_2={\rm d}z^2$. In such coordinates we have local measured foliations $(\mathcal F,\mu)=(\{\operatorname{Re}(z)={\rm const}\}, \, |{\rm d} \operatorname{Re} (z)|)$ which piece together to form the vertical measured foliation $\mathcal{F}(q_2)$ of $q_2$. The key step in showing the harmonic map compactification agrees with Thurston's compactification is to show that the length spectrum of~$h_t$ is asymptotically the same as the length spectrum of the vertical measured foliation of $t q_2$. That is, for any closed curve $\gamma$ on $\Sigma$, as $t\rightarrow \infty$,
\begin{gather}\label{rnk2lengthspec}
 l(f_t(\gamma))=l_{\gamma}(h_t)\sim {\rm i}(\mathcal{F}(tq_2),\gamma).
\end{gather}
Here, ${\rm i}(\mathcal F(tq_2),\gamma)$ is the intersection number of $\gamma$ with the vertical measured foliation $\mathcal F(tq_2)$. This compactification was further extended to the character variety for ${\rm SL}(2,\mathbb C)$ (see~\cite{Bestivina} and~\cite{DDW}).

As a generalization of the harmonic map compactification, we discuss the asymptotic behavior of the equivariant harmonic maps $f_t\colon \widetilde\Sigma\rightarrow N$ along the $\mathbb C^*$-family $t\cdot [(E,\phi)]\in \mathcal{M}_{\rm Higgs}({\rm SL}(n,\mathbb C))$ as $t\rightarrow \infty$ and aim to generalize the asymptotic formula~(\ref{rnk2lengthspec}).

For the left hand side of formula~\eqref{rnk2lengthspec}, we generalize the notion of length of a curve to a vector distance between two points in the target space $N$. For two Hermitian metrics $h_1$, $h_2$ on an $n$-dimensional complex vector space $V$, we can take a base $e_1,e_2,\dots,e_n$ of $V$ which is orthogonal with respect to both $h_1$ and $h_2$. We have the real numbers $k_j$ $(j=1,2,\dots,n)$ determined by $k_j=\log |e_j|_{h_2}-\log |e_j|_{h_1}$. We impose $k_1\geq k_2\geq \cdots\geq k_n$ and set $\vec{d}(h_1,h_2):=(k_1,\dots,k_n)\in \mathbb R^n$.

For the right hand side of formula~\eqref{rnk2lengthspec}, we use Higgs field to generalize holomorphic quadratic differential. Unfortunately we only have a local geometric object as a natural generalization of measured foliations. Therefore, instead of working with any closed curve on $\Sigma$, we restrict to consider ``nice" paths on the universal cover $\widetilde\Sigma$. Denote by $D(E,\phi)$ the set of points where the Higgs field $\phi$ fails to have $n$ distinct eigen $1$-forms, called the \textit{discriminant} of the Higgs bundle. Take a universal covering $\pi\colon Y\rightarrow \Sigma{\setminus} D(E,\phi)$, we have a decomposition of the Higgs bundle
\begin{gather*}\pi^*\big(E,\bar\partial_E,\phi\big)=\bigoplus\limits_{i=1}^n\big(E_i,\bar\partial_{E_i},\phi_i\cdot {\rm id}_{E_i}\big),\end{gather*} where $\phi_i$ are holomorphic $1$-forms, $\operatorname{rank} E_i=1$, and $\phi_i-\phi_j$ $(i\neq j)$ have no zeros. Let $\gamma\colon [0,1]$ $\rightarrow Y$ be a $C^{\infty}$-path, we have the expression $\gamma^*(\phi_i)=a_i(s) {\rm d}s$, where $a_i$ are $C^{\infty}$-functions on $[0,1]$. A path $\gamma$ is called \textit{non-critical} if $\operatorname{Re} (a_i(s))\neq \operatorname{Re} (a_j(s))$ $(i\neq j)$ for any $s\in [0,1]$. Let's reorder the~$a_i(s)$ such that $\operatorname{Re} (a_i(s))> \operatorname{Re} (a_j(s))$ for $i<j$ and set $\alpha_i:=-\int_0^1 \operatorname{Re}(a_i(s)){\rm d}s$. The vector $(\alpha_1,\dots, \alpha_n)$ generalizes the intersection number of the measured foliation.

With the above preparation, we finally state the following conjecture as a local generalization of the asymptotic formula (\ref{rnk2lengthspec}) to higher rank Higgs bundles.
\begin{Conjecture}[Hitchin WKB problem, Katzarkov--Noll--Pandit--Simpson \cite{HitchinWKB}]\label{HitchinWKBproblem} As $t\rightarrow \infty$, the harmonic map $f_t$ satisfies for a non-critical path $\gamma\colon [0,1]\rightarrow Y $,
\begin{gather*}
\frac{1}{t}\vec d(f_t(\gamma(0),f_t(\gamma(1)))\sim 2(\alpha_1,\dots,\alpha_n).\end{gather*}\end{Conjecture}

To answer the conjecture, we introduce the following notion.
\begin{Definition}
We call a Higgs bundle $(E, \phi)$ \textit{generically regular semisimple} if the discriminant set $D(E,\phi)$ is finite.
\end{Definition}
\begin{Remark}A ${\rm SL}(2,\mathbb C)$-Higgs bundle is either in the nilpotent cone or is generically regular semisimple. However, for $n\geq 3$, there are many ${\rm SL}(n,\mathbb C)$-Higgs bundles which are neither generically regular semisimple nor nilpotent.
\end{Remark}

\begin{Theorem}[Mochizuki \cite{Mochizuki}] \label{Asym1} Let $(E, \phi)$ be a stable Higgs bundle of degree~$0$ on~$\Sigma$. Suppose it is generically regular semisimple. If $\gamma\colon [0,1]\rightarrow Y$ is non-critical, then there exist positive constants $C_0$ and $\epsilon_0$ such that the following holds:
\begin{gather*}\left|\frac{1}{t}\vec d(f_t(\gamma(0),f_t(\gamma(1)))-2(\alpha_1,\dots,\alpha_n)\right|\leq C_0 \exp(-\epsilon_0 |t|).\end{gather*}
The constants $C_0$ and $\epsilon_0$ may depend only on $\Sigma, \phi_1,\dots,\phi_n$ and $\gamma$.
\end{Theorem}

The proof of Theorem \ref{Asym1} is based on the following key estimate as ``decoupling the Hitchin equation".
\begin{Theorem}[Mochizuki \cite{Mochizuki}] \label{Asym2}
Under the same assumptions in Theorem~{\rm \ref{Asym1}}. Then take any neighborhood $N_0$ of $D(E,\phi)$, there exists a constant $C_0>0$ and $\epsilon_0>0$ such that the following holds on $\Sigma{\setminus} N_0$,
\begin{gather*}
||F_{\nabla_{\bar\partial_E, H_t}}||=|t|^2||[\phi,\phi^{*_{H_t}}]||\leq C_0 \exp(-\epsilon_0 |t|).\end{gather*}
The constants $C_0$, $\epsilon_0$ only depend on $\Sigma$, $g_0$, $N_0$ and $(E ,\phi)$.\end{Theorem}

\begin{Remark}There is also another approach to obtain the decoupling phenomenon in Theorem~\ref{Asym2} for the Hitchin equation in~\cite{Fredrickson, MSWW16} for generic Higgs bundles, which will be addressed in the survey paper~\cite{ENDS} of L.~Fredrickson.
\end{Remark}

\begin{Remark}\quad\begin{enumerate}\itemsep=0pt
\item[(1)] For cyclic Higgs bundles in the Hitchin component, Theorems~\ref{Asym1} and~\ref{Asym2} were first proven in Loftin~\cite{LoftinAsym} for $n=3$ and in Collier--Li \cite{CollierLi} for $n>3$.
\item[(2)] The full picture of Conjecture~\ref{HitchinWKBproblem} remains open for Higgs bundles which are neither ge\-ne\-rically regular semisimple nor nilpotent. In fact, for those Higgs bundles which are not generically regular semisimple, Conjecture~\ref{HitchinWKBproblem} is not necessarily true. So we need a more refined description of the asymptotics for such families of Higgs bundles.
\end{enumerate}
\end{Remark}

\section{Negative curvature}
Suppose we are given a reductive representation $\rho$ into ${\rm SL}(n,\mathbb C)$, a Riemann surface structure $\Sigma$ and a $\rho$-equivariant harmonic map $f$ from $\widetilde\Sigma$ into $N\cong {\rm SL}(n,\mathbb C)/{\rm SU}(n)$. At an immersed point $x\in \Sigma$ of $f$, let $\sigma$ be the tangent plane at $f(x)$ tangential to the image $f\big(\widetilde\Sigma\big)$, $k_{\sigma}^N$ be the sectional curvature of $\sigma$, and $\kappa$ be the Gaussian curvature of the pullback metric at $x$.

In the case $n=2$, we have $k_{\sigma}^N=K_{{\rm SL}(2,\mathbb C)/{\rm SU}(2)}= -\frac{1}{2}$ at every immersed point and hence from Proposition \ref{Reducing}, the induced curvature $\kappa$ is no greater than $-\frac{1}{2}$. In the case $n\geq 3$, we cannot expect $k_{\sigma}^N$ to be negative because ${\rm SL}(n,\mathbb C)/{\rm SU}(n)$ contains an isometrically embedded Euclidean space of rank $n-1\geq 2$. However, for Hitchin representations into ${\rm PSL}(n,\mathbb R)$ or maximal representations into ${\rm Sp}(2n,\mathbb R)$, we expect $k_{\sigma}^N$ to be negative.
\begin{Conjecture}\label{NegativeCurvature}
If $\rho$ is a Hitchin representation into ${\rm PSL}(n,\mathbb R)$ or a maximal representation into ${\rm Sp}(2n,\mathbb R)$, then the sectional curvature $k_{\sigma}^N$ is negative at every immersed point.

As a result, the induced curvature $\kappa$ is negative at every immersed point.
\end{Conjecture}
Geometrically, the statement $k_{\sigma}^N<0$ in Conjecture \ref{NegativeCurvature} says that the image of the harmonic map $f$ is never tangential to a flat in the symmetric space $N$, that is, an isometrically embedded Euclidean space. The corollary $\kappa<0$ in Conjecture \ref{NegativeCurvature} follows immediately from $\kappa\leq k_{\sigma}^N$ in Proposition \ref{Reducing}.

\begin{Remark}When $\rho$ is a Hitchin representation into ${\rm PSL}(n,\mathbb R)$ and $f$ is a minimal mapping, Conjecture \ref{NegativeCurvature} is the negative curvature conjecture in Dai--Li \cite{DaiLi}. In this case, it is always a~minimal immersion since the Higgs field is non-vanishing everywhere.
\end{Remark}

Motivated by the formula (\ref{CurvatureForm}) of $k_{\sigma}^N$, we make the following conjecture which is slightly stronger than Conjecture \ref{NegativeCurvature}. \begin{Conjecture}\label{NeverDecouple}
If $[(E,\phi)]$ is a Higgs bundle in the Hitchin section or corresponds to a maximal ${\rm Sp}(2n,\mathbb R)$-representation, then the Hitchin equation never decouples:
\begin{gather*}
 F_{\nabla_{\bar\partial_E, H}}\neq 0,\qquad [\phi,\phi^{*_H}]\neq 0.\end{gather*}
\end{Conjecture}
Following from Remark \ref{Immerse}, Conjecture~\ref{NeverDecouple} holds if and only if both Conjecture~\ref{NegativeCurvature} holds and the harmonic map is an immersion.

\begin{Remark}As explained in Sections \ref{ExampleRank2}, \ref{ExampleRank3}, and \ref{ExampleRankN}, Conjecture~\ref{NeverDecouple} is true for Higgs bundles in the Hitchin component for ${\rm PSL}(2,\mathbb R)$, Higgs bundles in Gothen components for ${\rm Sp}(4,\mathbb R)$~\cite{DaiLi2} and cyclic ${\rm SL}(n,\mathbb C)$-Higgs bundles in the Hitchin section~\cite{DaiLi2}.
\end{Remark}

\begin{Remark}The phenomenon in Conjecture \ref{NeverDecouple} is opposite to the asymptotic behavior of Higgs bundles in which case the Hitchin equation asymptotically decouples in exponential decay as in Theorem~\ref{Asym2}.
\end{Remark}

\section[Monotonicity along $\mathbb C^*$-flow and the Hitchin fibration]{Monotonicity along $\boldsymbol{\mathbb C^*}$-flow and the Hitchin fibration}
Given a Higgs bundle $[(E,\phi)]$ in $\mathcal M_{\rm Higgs}({\rm SL}(n,\mathbb C))$ and denote by $f_{[(E,\phi)]}\colon \big(\widetilde S, \widetilde g_0\big)\rightarrow N$ the corresponding equivariant harmonic map. We consider the $\mathbb C^*$-family of Higgs bundles $t\cdot [(E,\phi)]$ and the corresponding equivariant harmonic maps $f_{t\cdot [(E,\phi)]}$.
\begin{Theorem}[Hitchin \cite{Hitchin87}] Along the $\mathbb C^*$-flow, the Morse function $($energy$)$ $E(f_{t\cdot [(E,\phi)]})$ decreases as $|t|$ decreases. \end{Theorem}

From the integrated monotonicity to the pointwise monotonicity, we make the following conjecture.
\begin{Conjecture}[\cite{LiHitchin}]\label{monotonicity}
Along the $\mathbb C^*$-flow, the energy density $e(f_{t\cdot [(E,\phi)]})$ decreases pointwise as $|t|$ decreases.
\end{Conjecture}
Dai and Li in \cite{DaiLi2} showed that Conjecture~\ref{monotonicity} holds for stable cyclic Higgs bundles, where the property of cyclic Higgs bundles in Lemma \ref{diagonal} is essentially used.

A weaker version of Conjecture \ref{monotonicity} is about the comparison between $(E,\phi)$ with the limit of $t\cdot [(E,\phi)]$ as $t\rightarrow 0$, which always lies in the nilpotent cone.
\begin{Conjecture}\label{DominateZero} For a Higgs bundle $[(E,\phi)]$ in $\mathcal{M}_{\rm Higgs}({\rm SL}(n,\mathbb C))$, the energy density satisfies
\begin{gather*}e(f_{[(E,\phi)]})\geq e\big(f_{\lim\limits_{t\rightarrow 0} t\cdot [(E,\phi)]}\big).\end{gather*}
\end{Conjecture}

Li in \cite{LiHitchin} showed that Conjecture \ref{DominateZero} holds for every Higgs bundle in the Hitchin section. One may check the sketch of the proof in Section~\ref{NonCyclic}.

We recall the definition of a Hitchin fiber in Section~\ref{Two}. If we stay in a single Hitchin fiber, we expect the maximum of the energy density to occur exactly at the image of the Hitchin section.
\begin{Conjecture}[Dai--Li \cite{DaiLi2}]\label{HitchinFiberMaximal}
Let $\big[\big(\tilde{E},\tilde{\phi}\big)\big]$ be a Higgs bundle in the Hitchin section and $[(E,\phi)]$ be a distinct polystable ${\rm SL}(n,\mathbb{C})$-Higgs bundle in the same Hitchin fiber. Then the corresponding harmonic maps satisfy $e(f_{[(E,\phi)]})<e(f_{[(\tilde E,\tilde \phi)]})$ and hence $f_{[(E,\phi)]}^*g_N<f_{[(\tilde E,\tilde \phi)]}^*g_N$.

As a result, the Morse function $($energy$)$ satisfies $E(f_{[(E,\phi)]})<E(f_{[(\tilde E,\tilde \phi)]})$.
\end{Conjecture}

In the case $n=2$, Conjecture~\ref{HitchinFiberMaximal} is shown by Deroin and Tholozan in~\cite{DominationFuchsian}.

In the case $n\geq 3$, even the integrated version of Conjecture \ref{HitchinFiberMaximal} remains open.
\begin{Conjecture}\label{MorseFiber}Inside each Hitchin fiber of $\mathcal M_{\rm Higgs}({\rm SL}(n,\mathbb C))$, the maximum of the Morse function $($energy$)$ occurs exactly at the image of the Hitchin section.
\end{Conjecture}
If Conjecture \ref{MorseFiber} holds, one can define the Hitchin section intrinsically as the only maximum of the Morse function inside each Hitchin fiber instead of the explicit construction in the form~(\ref{HitchinSection}).

In the end of this section, let's explain Conjectures~\ref{DominateZero} and~\ref{HitchinFiberMaximal} in terms of the following picture:
\begin{figure}[h!]\centering
 \includegraphics{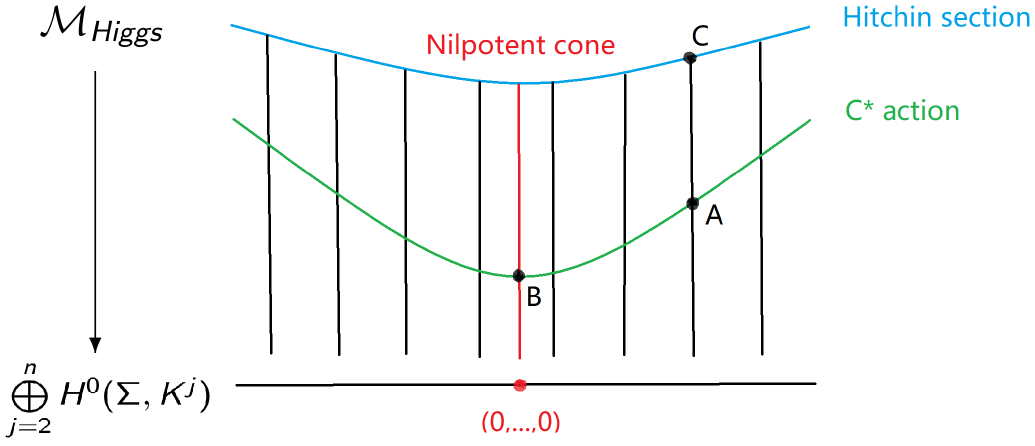}
\end{figure}

In the above picture, for the point $A\in \mathcal M_{\rm Higgs}({\rm SL}(n,\mathbb C))$, one can immediately determine the point~$B$ to be the limit of the $\mathbb C^*$-flow $t\cdot A$ as $t\rightarrow 0$ in the nilpotent cone and the point~$C$ to be the intersection point of the Hitchin fiber containing $A$ and the Hitchin section. Then Conjectures~\ref{DominateZero} and~\ref{HitchinFiberMaximal} say that the energy densities of the corresponding harmonic maps $f_A, f_B, f_C\colon \widetilde\Sigma\rightarrow N$ satisfy \begin{gather*}e(f_B)<e(f_A)< e(f_C).\end{gather*}

\subsection*{Acknowledgement}
The author would like to thank the anonymous referees for numerous suggestions and comments to help improving the manuscript. The author also thanks Laura Schaposnik for her kind invitation to give the mini-course in the RTG workshop in June 2018 at UIC, and both her and Lara Anderson for the encouragement to write these notes. The author acknowledges support from UIC NSF RTG Grant DMS-1246844, the UIC Start-Up Fund of Laura Schaposnik, and the grants NSF DMS 1107452, 1107263, 1107367 RNMS: GEometric structures And Representation varieties (the GEAR Network). The author acknowledges support from Nankai Zhide Foundation.

\pdfbookmark[1]{References}{ref}
\LastPageEnding


\begin{thebibliography}{99}
\footnotesize\itemsep=0pt

\bibitem{AlessandriniCollier}
Alessandrini D., Collier B., The geometry of maximal components of the {${\rm
 PSp}(4,{\mathbb R})$} character variety, \textit{Geom. Topol.}, {t}o appear,
 \href{https://arxiv.org/abs/1708.05361}{arXiv:1708.05361}.

\bibitem{Bar}
Baraglia D., ${G}_2$ geometry and integrable systems, Ph.D.~Thesis,
 {U}niversity of Oxford, 2009, \href{https://arxiv.org/abs/1002.1767}{arXiv:1002.1767}.

\bibitem{Bar1}
Baraglia D., Cyclic {H}iggs bundles and the affine {T}oda equations,
 \href{https://doi.org/10.1007/s10711-014-0003-2}{\textit{Geom. Dedicata}} \textbf{174} (2015), 25--42, \href{https://arxiv.org/abs/1011.6421}{arXiv:1011.6421}.

\bibitem{Bestivina}
Bestvina M., Degenerations of the hyperbolic space, \href{https://doi.org/10.1215/S0012-7094-88-05607-4}{\textit{Duke Math.~J.}}
 \textbf{56} (1988), 143--161.

\bibitem{BradlowDeformation}
Bradlow S.B., Garc\'{\i}a-Prada O., Gothen P.B., Deformations of maximal
 representations in {${\rm Sp}(4,\mathbb{R})$}, \href{https://doi.org/10.1093/qmath/har010}{\textit{Q.~J.~Math.}}
 \textbf{63} (2012), 795--843, \href{https://arxiv.org/abs/0903.5496}{arXiv:0903.5496}.

\bibitem{Burger}
Burger M., Iozzi A., Labourie F., Wienhard A., Maximal representations of
 surface groups: symplectic {A}nosov structures, \href{https://doi.org/10.4310/PAMQ.2005.v1.n3.a5}{\textit{Pure Appl. Math.~Q.}}
 \textbf{1} (2005), 543--590, \href{https://arxiv.org/abs/math.DG/0506079}{arXiv:math.DG/0506079}.

\bibitem{Burger1}
Burger M., Iozzi A., Wienhard A., Surface group representations with maximal
 {T}oledo invariant, \href{https://doi.org/10.4007/annals.2010.172.517}{\textit{Ann. of Math.}} \textbf{172} (2010), 517--566,
 \href{https://arxiv.org/abs/math.DG/0605656}{arXiv:math.DG/0605656}.

\bibitem{Collier}
Collier B., Maximal {${\rm Sp}(4,\mathbb{R})$} surface group representations,
 minimal immersions and cyclic surfaces, \href{https://doi.org/10.1007/s10711-015-0101-9}{\textit{Geom. Dedicata}} \textbf{180}
 (2016), 241--285, \href{https://arxiv.org/abs/1503.03526}{arXiv:1503.03526}.

\bibitem{CollierLi}
Collier B., Li Q., Asymptotics of {H}iggs bundles in the {H}itchin component,
 \href{https://doi.org/10.1016/j.aim.2016.11.031}{\textit{Adv. Math.}} \textbf{307} (2017), 488--558, \href{https://arxiv.org/abs/1405.1106}{arXiv:1405.1106}.

\bibitem{CollierTholozanToulisse}
Collier B., Nicolas T., Toulisse J., The geometry of maximal representations of
 surface groups into ${\rm SO}(2,n)$, \textit{Duke Math.~J.}, {t}o
 appear, \href{https://arxiv.org/abs/1702.08799}{arXiv:1702.08799}.

\bibitem{Corlette}
Corlette K., Flat {$G$}-bundles with canonical metrics, \href{https://doi.org/10.4310/jdg/1214442469}{\textit{J.~Differential
 Geom.}} \textbf{28} (1988), 361--382.

\bibitem{DaiLi}
Dai S., Li Q., Minimal surfaces for {H}itchin representations,
 \textit{J.~Differential Geom.}, {t}o appear, \href{https://arxiv.org/abs/1605.09596}{arXiv:1605.09596}.

\bibitem{DaiLi2}
Dai S., Li Q., On cyclic {H}iggs bundles, \href{https://doi.org/10.1007/s00208-018-1779-4}{\textit{Math. Ann.}}, {t}o
 appear, \href{https://arxiv.org/abs/1710.10725}{arXiv:1710.10725}.

\bibitem{DDW}
Daskalopoulos G., Dostoglou S., Wentworth R., On the {M}organ--{S}halen
 compactification of the {${\rm SL}(2,{\bf C})$} character varieties of
 surface groups, \href{https://doi.org/10.1215/S0012-7094-00-10121-4}{\textit{Duke Math.~J.}} \textbf{101} (2000), 189--207,
 \href{https://arxiv.org/abs/math.DG/9810034}{arXiv:math.DG/9810034}.

\bibitem{DominationFuchsian}
Deroin B., Tholozan N., Dominating surface group representations by {F}uchsian
 ones, \href{https://doi.org/10.1093/imrn/rnv275}{\textit{Int. Math. Res. Not.}} \textbf{2016} (2016), 4145--4166,
 \href{https://arxiv.org/abs/1311.2919}{arXiv:1311.2919}.

\bibitem{Donaldson}
Donaldson S.K., Twisted harmonic maps and the self-duality equations,
 \href{https://doi.org/10.1112/plms/s3-55.1.127}{\textit{Proc. London Math. Soc.}} \textbf{55} (1987), 127--131.

\bibitem{DumasWolf}
Dumas D., Wolf M., Polynomial cubic differentials and convex polygons in the
 projective plane, \href{https://doi.org/10.1007/s00039-015-0344-5}{\textit{Geom. Funct. Anal.}} \textbf{25} (2015), 1734--1798,
 \href{https://arxiv.org/abs/1407.8149}{arXiv:1407.8149}.

\bibitem{EellsSampson}
Eells Jr. J., Sampson J.H., Harmonic mappings of {R}iemannian manifolds,
 \href{https://doi.org/10.2307/2373037}{\textit{Amer.~J. Math.}} \textbf{86} (1964), 109--160.

\bibitem{FLP}
Fathi A., Laudenbach F., Po\'{e}naru V., Thurston's work on surfaces,
 \textit{Mathematical Notes}, Vol.~48, Princeton University Press, Princeton,
 NJ, 2012.

\bibitem{Fredrickson}
Fredrickson L., Generic ends of the moduli space of ${\rm SL}(n,\mathbb
 C)$-Higgs bundles, \href{https://arxiv.org/abs/1810.01556}{arXiv:1810.01556}.

\bibitem{ENDS}
Fredrickson L., Perspectives on the asymptotic geometry of the {H}itchin moduli
 space, \href{https://doi.org/10.3842/SIGMA.2019.018}{\textit{SIGMA}} \textbf{15} (2019), 018, 20~pages, \href{https://arxiv.org/abs/1809.05735}{arXiv:1809.05735}.

\bibitem{GGM}
Garcia-Prada O., Gothen P.B., Mundet~i Riera I., The {H}itchin--{K}obayashi
 correspondence, {H}iggs pairs and surface group representations,
 \href{https://arxiv.org/abs/0909.448}{arXiv:0909.448}.

\bibitem{Gothen}
Gothen P.B., Components of spaces of representations and stable triples,
 \href{https://doi.org/10.1016/S0040-9383(99)00086-5}{\textit{Topology}} \textbf{40} (2001), 823--850, \href{https://arxiv.org/abs/math.AG/9904114}{arXiv:math.AG/9904114}.

\bibitem{Guest}
Guest M.A., Harmonic maps, loop groups, and integrable systems, \textit{London
 Mathematical Society Student Texts}, Vol.~38, \href{https://doi.org/10.1017/CBO9781139174848}{Cambridge University Press},
 Cambridge, 1997.

\bibitem{LinGuest}
Guest M.A., Lin C.S., Nonlinear {PDE} aspects of the tt* equations of {C}ecotti
 and {V}afa, \href{https://doi.org/10.1515/crelle-2012-0057}{\textit{J.~Reine Angew. Math.}} \textbf{689} (2014), 1--32,
 \href{https://arxiv.org/abs/1010.1889}{arXiv:1010.1889}.

\bibitem{Guichard}
Guichard O., An introduction to the differential geometry of flat bundles and
 of {H}iggs bundles, in The Geometry, Topology and Physics of Moduli Spaces of
 {H}iggs Bundles, \textit{Lect. Notes Ser. Inst. Math. Sci. Natl. Univ.
 Singap.}, Vol.~36, \href{https://doi.org/10.1142/9789813229099_0001}{World Sci. Publ.}, Hackensack, NJ, 2018, 1--63.

\bibitem{Hitchin87}
Hitchin N.J., The self-duality equations on a {R}iemann surface, \href{https://doi.org/10.1112/plms/s3-55.1.59}{\textit{Proc.
 London Math. Soc.}} \textbf{55} (1987), 59--126.

\bibitem{Hitchin92}
Hitchin N.J., Lie groups and {T}eichm\"{u}ller space, \href{https://doi.org/10.1016/0040-9383(92)90044-I}{\textit{Topology}}
 \textbf{31} (1992), 449--473.

\bibitem{Jost}
Jost J., Riemannian geometry and geometric analysis, 3rd~ed., \textit{Universitext},
 \href{https://doi.org/10.1007/978-3-662-04672-2}{Springer-Verlag}, Berlin, 2002.

\bibitem{MaximumPrinciple}
Jost J., Partial differential equations, 2nd~ed., \textit{Graduate Texts in
 Mathematics}, Vol.~214, \href{https://doi.org/10.1007/978-0-387-49319-0}{Springer}, New York, 2007.

\bibitem{HitchinWKB}
Katzarkov L., Noll A., Pandit P., Simpson C., Harmonic maps to buildings and
 singular perturbation theory, \href{https://doi.org/10.1007/s00220-014-2276-6}{\textit{Comm. Math. Phys.}} \textbf{336} (2015),
 853--903, \href{https://arxiv.org/abs/1311.7101}{arXiv:1311.7101}.

\bibitem{Kobayashi}
Kobayashi S., Differential geometry of complex vector bundles,
\textit{Publications of the Mathematical Society of Japan}, Vol.~15,
 \href{https://doi.org/10.1515/9781400858682}{Princeton University Press}, Princeton, NJ, 1987.

\bibitem{LabourieAnosov}
Labourie F., Anosov flows, surface groups and curves in projective space,
 \href{https://doi.org/10.1007/s00222-005-0487-3}{\textit{Invent. Math.}} \textbf{165} (2006), 51--114,
 \href{https://arxiv.org/abs/math.DG/0401230}{arXiv:math.DG/0401230}.

\bibitem{LabourieFlat}
Labourie F., Flat projective structures on surfaces and cubic holomorphic
 differentials, \href{https://doi.org/10.4310/PAMQ.2007.v3.n4.a10}{\textit{Pure Appl. Math.~Q.}} \textbf{3} (2007), 1057--1099,
 \href{https://arxiv.org/abs/math.DG/0611250}{arXiv:math.DG/0611250}.

\bibitem{LabourieEnergy}
Labourie F., Cross ratios, {A}nosov representations and the energy functional
 on {T}eichm\"{u}ller space, \href{https://doi.org/10.24033/asens.2072}{\textit{Ann. Sci. \'{E}c. Norm. Sup\'{e}r.~(4)}}
 \textbf{41} (2008), 437--469, \href{https://arxiv.org/abs/math.DG/0512070}{arXiv:math.DG/0512070}.

\bibitem{LabourieCyclic}
Labourie F., Cyclic surfaces and {H}itchin components in rank~2, \href{https://doi.org/10.4007/annals.2017.185.1.1}{\textit{Ann.
 of Math.}} \textbf{185} (2017), 1--58, \href{https://arxiv.org/abs/1406.4637}{arXiv:1406.4637}.

\bibitem{LiHitchin}
Li Q., Harmonic maps for {H}itchin representations, \href{https://doi.org/10.1007/s00039-019-00491-7}{\textit{Geom. Funct. Anal.}}
 \textbf{29} (2019), 539--560, \href{https://arxiv.org/abs/1806.06884}{arXiv:1806.06884}.

\bibitem{LoftinAsym}
Loftin J., Flat metrics, cubic differentials and limits of projective
 holonomies, \href{https://doi.org/10.1007/s10711-007-9184-2}{\textit{Geom. Dedicata}} \textbf{128} (2007), 97--106,
 \href{https://arxiv.org/abs/math.DG/0611289}{arXiv:math.DG/0611289}.

\bibitem{LoftinSurvey}
Loftin J., Survey on affine spheres, in Handbook of Geometric Analysis,
 {N}o.~2, \textit{Adv. Lect. Math. (ALM)}, Vol.~13, Int. Press, Somerville,
 MA, 2010, 161--191, \href{https://arxiv.org/abs/0809.1186}{arXiv:0809.1186}.

\bibitem{Loftin}
Loftin J.C., Affine spheres and convex {$\mathbb{RP}^n$}-manifolds,
 \href{https://doi.org/10.1353/ajm.2001.0011}{\textit{Amer.~J. Math.}} \textbf{123} (2001), 255--274.

\bibitem{MSWW16}
Mazzeo R., Swoboda J., Weiss H., Witt F., Ends of the moduli space of {H}iggs
 bundles, \href{https://doi.org/10.1215/00127094-3476914}{\textit{Duke Math.~J.}} \textbf{165} (2016), 2227--2271,
 \href{https://arxiv.org/abs/1405.5765}{arXiv:1405.5765}.

\bibitem{Mochizuki}
Mochizuki T., Asymptotic behaviour of certain families of harmonic bundles on
 {R}iemann surfaces, \href{https://doi.org/10.1112/jtopol/jtw018}{\textit{J.~Topol.}} \textbf{9} (2016), 1021--1073,
 \href{https://arxiv.org/abs/1508.05997}{arXiv:1508.05997}.

\bibitem{MorganShalenCompact}
Morgan J.W., Shalen P.B., Valuations, trees, and degenerations of hyperbolic
 structures.~{I}, \href{https://doi.org/10.2307/1971082}{\textit{Ann. of Math.}} \textbf{120} (1984), 401--476.

\bibitem{Parreau}
Parreau A., Compactification d'espaces de repr\'{e}sentations de groupes de
 type fini, \href{https://doi.org/10.1007/s00209-011-0921-8}{\textit{Math.~Z.}} \textbf{272} (2012), 51--86, \href{https://arxiv.org/abs/1003.1111}{arXiv:1003.1111}.

\bibitem{Reznikov}
Reznikov A.G., Harmonic maps, hyperbolic cohomology and higher {M}ilnor
 inequalities, \href{https://doi.org/10.1016/0040-9383(93)90056-2}{\textit{Topology}} \textbf{32} (1993), 899--907.

\bibitem{SacksUhlenbeck1}
Sacks J., Uhlenbeck K., The existence of minimal immersions of {$2$}-spheres,
 \href{https://doi.org/10.2307/1971131}{\textit{Ann. of Math.}} \textbf{113} (1981), 1--24.

\bibitem{SacksUhlenbeck2}
Sacks J., Uhlenbeck K., Minimal immersions of closed {R}iemann surfaces,
 \href{https://doi.org/10.2307/1998902}{\textit{Trans. Amer. Math. Soc.}} \textbf{271} (1982), 639--652.

\bibitem{Sampson}
Sampson J.H., Some properties and applications of harmonic mappings,
 \href{https://doi.org/10.24033/asens.1344}{\textit{Ann. Sci. \'{E}cole Norm. Sup.~(4)}} \textbf{11} (1978), 211--228.

\bibitem{SchoenYauUnivalent}
Schoen R., Yau S.T., On univalent harmonic maps between surfaces,
 \href{https://doi.org/10.1007/BF01403164}{\textit{Invent. Math.}} \textbf{44} (1978), 265--278.

\bibitem{SchoenYau}
Schoen R., Yau S.T., Existence of incompressible minimal surfaces and the
 topology of three-dimensional manifolds with nonnegative scalar curvature,
 \href{https://doi.org/10.2307/1971247}{\textit{Ann. of Math.}} \textbf{110} (1979), 127--142.

\bibitem{Schoen}
Schoen R.M., The role of harmonic mappings in rigidity and deformation
 problems, in Complex Geometry ({O}saka, 1990), \textit{Lecture Notes in Pure
 and Appl. Math.}, Vol.~143, Dekker, New York, 1993, 179--200.

\bibitem{SimonWang}
Simon U., Wang C.P., Local theory of affine {$2$}-spheres, in Differential
 Geometry: {R}iemannian Geometry ({L}os {A}ngeles, {CA}, 1990), \textit{Proc.
 Sympos. Pure Math.}, Vol.~54, Amer. Math. Soc., Providence, RI, 1993,
 585--598.

\bibitem{Simpson88}
Simpson C.T., Constructing variations of {H}odge structure using
 {Y}ang--{M}ills theory and applications to uniformization, \href{https://doi.org/10.2307/1990994}{\textit{J.~Amer.
 Math. Soc.}} \textbf{1} (1988), 867--918.

\bibitem{ThurstonCompactification}
Thurston W.P., On the geometry and dynamics of diffeomorphisms of surfaces,
 \href{https://doi.org/10.1090/S0273-0979-1988-15685-6}{\textit{Bull. Amer. Math. Soc. (N.S.)}} \textbf{19} (1988), 417--431.

\bibitem{Wang}
Wang C.P., Some examples of complete hyperbolic affine {$2$}-spheres in {${\mathbb R}^3$}, in Global Differential Geometry and Global Analysis ({B}erlin, 1990),
 \textit{Lecture Notes in Math.}, Vol.~1481, \href{https://doi.org/10.1007/BFb0083648}{Springer}, Berlin, 1991, 271--280.

\bibitem{Wentworth}
Wentworth R.A., Higgs bundles and local systems on {R}iemann surfaces, in
 Geometry and Quantization of Moduli Spaces, \textit{Adv. Courses Math. CRM Barcelona},
 \href{https://doi.org/10.1007/978-3-319-33578-0_4}{Birkh\"{a}user/Springer}, Cham, 2016, 165--219, \href{https://arxiv.org/abs/1402.4203}{arXiv:1402.4203}.

\bibitem{TeichOfHarmonic}
Wolf M., The {T}eichm\"{u}ller theory of harmonic maps, \href{https://doi.org/10.4310/jdg/1214442885}{\textit{J.~Differential
 Geom.}} \textbf{29} (1989), 449--479.

\end{thebibliography}
\end{document}